\documentclass[12pt,a4paper]{article}
\usepackage{amsfonts}
\usepackage{subfigure}
\usepackage{graphicx}
\usepackage{comment}
\usepackage{color}
\usepackage{mathrsfs}
\usepackage{multirow}
\usepackage{mathtools}
\usepackage{stmaryrd}
\usepackage{amsmath}
\usepackage{amssymb}
\usepackage{bbm}      
\usepackage{mathrsfs}
\usepackage{graphicx}
\usepackage{algorithm}
\usepackage{algorithmic}
\usepackage{subfigure}
\usepackage{pgfplots}
\usepackage{mathrsfs}
\usetikzlibrary{arrows}
\usepackage{enumerate}
\usepackage{theorem}
\usepackage{ulem}
\usepackage{quiver}
\makeatletter
\def\tank#1{\protected@xdef\@thanks{\@thanks
		\protect\footnotetext[0]{#1}}}
\def\bigfoot{
	
	\@footnotetext}
\makeatother

\newcommand{\ea}{\end{array}}
\newtheorem{theorem}{Theorem}[section]

\newtheorem{corollary}{Corollary}[section]
\newtheorem{lemma}{Lemma}[section]
\newtheorem{definition}{Definition}[section]
\newtheorem{remark}{Remark}[section]
\newtheorem{example}{Example}[section]

{\theorembodyfont{\rmfamily}
}

\newenvironment{proof}{Proof.}

\pgfplotsset{compat=1.18}
\begin{document}  

\title{\Large\bf Dimension Calculation for Spline Spaces over Rectilinear Partitions via Smoothing Cofactor Method} 

\author{Bingru~Huang\thanks{Corresponding author Email: hbr999@ustc.edu.cn} \quad Falai~Chen \\
        School of Mathematical Sciences \\
        University of Science and Technology of China \\
        Hefei, 230026, People's Republic of China}
\date{}
\maketitle
\begin{center}
	\begin{minipage}{160mm}
{\bf Abstract.}  This paper presents a general framework for calculating the dimension of spline spaces over arbitrary rectilinear partitions using the smoothing cofactor method. The approach extends existing dimension theory for polynomial splines over T-meshes, reducing the problem to the rank computation of conformality matrices associated with $TE$-connected components. Furthermore, a new class of rectilinear partitions, termed partitions with disjoint truncated $l$-edges, is introduced. It is proven that under specific conditions, the dimension of the corresponding spline space attains Schumaker’s lower bound, demonstrating that this lower bound is attainable for arbitrary degree $d$ and smoothness order $\mu$ in certain partition configurations. Numerical examples, including the Morgan--Scott and Yuan--Stillman partitions, validate the effectiveness and generality of the framework for both triangular and non-triangular rectilinear partitions.
	\end{minipage}
\end{center}
\section{Introduction}
Splines, as piecewise smooth polynomials, are fundamental in approximation theory~\cite{deBoor2001,schumaker2007spline,trefethen2013}, numerical analysis~\cite{cheney2009,rivlin1981introduction}, geometric modeling~\cite{piegl1997,hoschek1993,farin2002}, computer graphics~\cite{foley1996}, and isogeometric analysis~\cite{hughes2005,cottrell2009,scott2011}. Important research problems for a space spline include determining the dimension of the spline space and constructing its basis functions. In most practical settings, these spline spaces are defined over rectilinear partitions, most commonly triangular or rectangular partitions.
A particularly fundamental and long-standing challenge in this field is the accurate computation of the dimension of spline spaces over such partitions. Three main approaches have been developed to address this problem:
\begin{itemize}
\item The first method constructs linear functionals $\lambda_i$, $i\in\mathcal{I}$ for index set $\mathcal{I}$, such that $\lambda_i(s)=0$ for all $i\in\mathcal{I}$ and $s \in S_d^\mu(\Delta)$ implies $s \equiv 0$. Schumaker~\cite{schumaker1979dimension} first provided a lower bound for spline spaces over arbitrary triangulations using this approach. For rectilinear partitions, lower bounds were established in~\cite{schumaker1984bounds,manni1991dimension}. This method is widely applied; a prominent example is the Bézier--Bernstein technique (or B-net method), which uses the Bernstein basis for spline patches to recast smoothness conditions across adjacent cells as a linear system on Bézier ordinates. Related works include~\cite{deng2000note,manni2018dimension,shi1991singularity,alfeld1987dimension,alfeld1987minimally,alfeld1992dimension,MDS,dim2006}. See Section 3 of~\cite{nurnberger2000developments} for details.
\item The second method, the smoothing cofactor method, applies Bézout's theorem from algebraic geometry to convert inter-cell continuity into smooth cofactor conformality conditions for dimension analysis. Introduced by Wang~\cite{wang1975structural}, it has been used extensively for spline spaces over triangular meshes, rectangular meshes, and T-meshes~\cite{wang2013multivariate,chui1983multivariate,chui1983smooth,wang1985dimension,chui1984spaces,huang2024stability,dim20061,dim2016,zeng2015,zeng2016,zeng2018,bracco2019tchebycheffian}.
\item The third method, the homological algebraic method, was pioneered by Billera~\cite{billera1988homology}, earning the Fulkerson Prize. Subsequent works~\cite{schenck1997local,schenck1997family,dim2014,bracco2016dimension,bracco2016generalized,yuan2019counter,yuan2019upper,toshniwal2021polynomial,toshniwal2020dimension,toshniwal2021counting,toshniwal2023algebraic,schenck2020new,mourrain2013homological,schenck2016algebraic} employ algebraic geometry tools for spline analysis. This method has advanced conjectures on spline dimensions; for instance, Yuan and Stillman~\cite{yuan2019counter} recently provided a counterexample to the Schenck--Stiller ``$2r+1$'' conjecture, which posits that for degree $\geq 2r+1$ (with smoothness order $r$), the dimension equals Schumaker's lower bound~\cite{schumaker1979dimension}. 
\end{itemize}

Despite significant progress through these methods, a key difficulty persists: in many cases, the dimension of a spline space over a rectilinear partition depends not only on topological (combinatorial) information — such as the number of vertices, edges, and cells — but also on the specific geometric configuration of the partition. This geometric dependence can lead to dimensional instability, where the dimension is not constant for all generic partitions sharing the same topology.

A classical illustration of this instability is the Morgan–Scott triangulation $  \Delta_{\mathrm{MS}}  $ (see Figure~\ref{fig:MS partition}, left). Morgan and Scott~\cite{morgan1977dimension} showed that the dimension of the biquadratic $  C^1  $ spline space $  S_2^1(\Delta_{\mathrm{MS}})  $ depends on the geometric positions of the vertices. Diener~\cite{diener1990instability} extended this result to the spaces $  S_{2r}^r(\Delta_{\mathrm{MS}})  $ for $  r \geq 2  $, confirming similar geometric dependence. Consequently, the dimension cannot in general be expressed solely in terms of topological quantities. Further investigations of this instability and related dimension problems appear in~\cite{li2011dimensions,feng1996dimension,deng2000note,shi1991singularity,alfeld1986dimension,Ins2012,li2019instability,guo2015problem}.

T-meshes, a special subclass of rectilinear partitions that support local refinement~\cite{ts1,ts2}, typically rely on piecewise tensor-product polynomial spaces~\cite{dim2006}. Recent research has clarified the structure of T-meshes~\cite{huang2024stability,huang2025dimension,huang2025preliminarystudydimensionalstability,zhong2025basis}, identified sources of dimensional instability at highest smoothness levels~\cite{huang2024stability}, and constructed bases for arbitrary T-meshes~\cite{zhong2025basis}. A key technique in these works applies the smoothing cofactor method to reduce dimension computation to determining the rank of the conformality matrix associated with each T-connected component~\cite{zeng2015}.

This paper generalizes the dimension computation framework for T-meshes to arbitrary rectilinear partitions. Analogous to the T-connected component, we introduce the \textbf{TE-connected component} to determine the dimension of the corresponding spline space.

The main contributions of this paper are as follows:
\begin{itemize}
\item[(1)] We extend the smoothing cofactor framework from T-meshes to arbitrary rectilinear partitions by introducing the concept of TE-connected components. This allows us to derive an explicit formula for the dimension of the spline space \(S_d^\mu(\Delta)\) in terms of the rank of an explicitly constructible conformality matrix \(M(TE(\Delta))\).

\item[(2)] We propose a complete, step-by-step algorithm for constructing the conformality matrix \(M(TE(\Delta))\) for general smoothness \(\mu\), based on homogeneous polynomial decomposition and decoupled edge cofactors.

\item[(3)] We introduce a new family of rectilinear partitions, called \textbf{partitions with disjoint truncated $l$-edges}. For this family, we prove that under a specific condition, the dimension of the corresponding spline space \(S_d^\mu(\Delta)\) exactly attains the lower bound proposed by Schumaker~\cite{schumaker1979dimension} for arbitrary degree $d$ and smoothness order $\mu$.

\item[(4)] We demonstrate the effectiveness and generality of the proposed framework through several examples. First, we validate our framework on two well-known benchmark examples: the Morgan--Scott partition, where the known dimensional instability is recovered, and the Yuan--Stillman partition, which provides a counterexample to the Schenck--Stiller ``$2r+1$'' conjecture. Second, we show that our method is not limited to triangular partitions, but can compute the dimension of spline spaces over arbitrary rectilinear partitions (including those with polygonal cells). 

\end{itemize}
The paper is organized as follows. Section~2 reviews the smoothing cofactor method for splines over T-meshes and the global conformality conditions for rectilinear partitions. Section~3 presents the general framework for computing the dimension of spline spaces over arbitrary rectilinear partitions. Section~4 introduces and analyzes a new family of rectilinear partitions with disjoint truncated $l$-edges. Section~5 demonstrates the effectiveness of the proposed framework through several representative examples. Finally, Section~6 concludes the paper and discusses directions for future work.

\section{Preliminaries}\label{sec.prelim}
In this section, we introduce the framework for computing the dimension of spline spaces over T-meshes using the smoothing cofactor method, with emphasis on the T-connected component. We also review key notions of rectilinear partitions and results on dimension computation over arbitrary rectilinear partitions via the smoothing cofactor method.

\subsection{Spline spaces over T-meshes}
We first introduce necessary terminology and notation. A T-mesh is a rectangular partition of a rectangular domain that allows T-junctions. Throughout this paper, we consider T-meshes with simply-connected interior faces and no holes. Vertices are the grid points of the mesh; among them, T-junctions are interior vertices of valence 3. An edge is a line segment connecting two adjacent vertices along a horizontal or vertical grid line. Edges are classified as boundary edges or interior edges according to their location.

A large edge (denoted $l$-edge) is a maximal straight segment composed of one or more collinear edges, with endpoints being either T-junctions or boundary vertices. An interior $l$-edge is further categorized as follows:
\begin{itemize}
    \item A cross-cut if both endpoints lie on the domain boundary.
    \item A T $l$-edge if both endpoints are T-junctions.
    \item A ray otherwise (one endpoint is a T-junction and the other lies on the boundary).
\end{itemize}

The set of all T $l$-edges in a T-mesh $\mathscr{T}$ forms the \textbf{T-connected component} denoted as $T(\mathscr{T})$, which is central to the dimension formula derived by the smoothing cofactor method.

We next define the spline space on a T-mesh. Let $\mathscr{F}$ denote the set of all rectangular faces. The spline space is given by
$$S_{d_1,d_2}^{\mu_1,\mu_2}(\mathscr{T}) = \left\{ s(x,y) \in C^{\mu_1,\mu_2}(\Phi) : s|_{\phi} \in \mathbb{P}_{d_1} \otimes \mathbb{P}_{d_2},\ \forall \phi \in \mathscr{F} \right\},$$
where $\mathbb{P}_{d_1} \otimes \mathbb{P}_{d_2}$ is the tensor-product polynomial space of bidegree $(d_1,d_2)$, and $C^{\mu_1,\mu_2}(\Phi)$ consists of bivariate functions that are $\mu_1$-times continuously differentiable in the $x$-direction and $\mu_2$-times in the $y$-direction over the domain $\Phi$ occupied by $\mathscr{F}$. This subsection focuses on the uniform highest-smoothness case: $d_1 = d_2 = d$ and $\mu_1 = \mu_2 = d-1$. The spline space is simply denoted as $S_d(\mathscr T)$ in this case.

The smoothing cofactor method, introduced in~\cite{wang1975structural} and extended to T-meshes in~\cite{huang2024stability,dim20061,dim2016,zeng2015,zeng2016,zeng2018,huang2025dimension,huang2025preliminarystudydimensionalstability,wu2013}, converts $C^{d-1,d-1}$ continuity constraints across interior edges into linear relations among vertex cofactors on each T $l$-edge(\cite{huang2024stability}).
Consider a horizontal T $l$-edge with vertices at $x$-coordinates $x_1 < x_2 < \cdots < x_r$. The associated \textbf{vertex cofactors} $\delta_1,\delta_2,\dots,\delta_r \in \mathbb{R}$ (see Figure~\ref{fig: vertex cofactors}) satisfy the conformality conditions~\cite{huang2024stability,dim20061,dim2016,wu2013}.
\begin{equation}\label{eq4}
\sum\limits_{i=1}^{r}\delta_{i}(x-x_i)^d=0
\end{equation}

\begin{figure}[htbp]
    \centering
\definecolor{black}{rgb}{0.26666666666666666,0.26666666666666666,0.26666666666666666}
\begin{tikzpicture}[line cap=round,line join=round,>=triangle 45,x=1cm,y=1cm,scale=2]
\draw [line width=1pt] (1,0.5)-- (1,2.5);
\draw [line width=1pt] (1,1.5)-- (5,1.5);
\draw [line width=1pt] (5,2.5)-- (5,0.5);
\draw [line width=1pt] (1.5,1.5)-- (1.5,2);
\draw [line width=1pt] (2,1.5)-- (2,1);
\draw [line width=1pt] (2.5,2)-- (2.5,1);
\draw [line width=1pt] (3,1.5)-- (3,1);
\draw [line width=1pt] (4.5,1.5)-- (4.5,1);
\draw [line width=1pt] (4,1.5)-- (4,2);
\draw (0.94,1.5) node[anchor=north west] {$\delta_1$};
\draw (1.35,1.5) node[anchor=north west] {$\delta_2$};
\draw (1.85,1.85) node[anchor=north west] {$\delta_3$};
\draw (2.48,1.5) node[anchor=north west] {$\delta_4$};
\draw (2.94,1.5) node[anchor=north west] {$\delta_5$};
\draw (3.82,1.5) node[anchor=north west] {$\delta_{r-2}$};
\draw (4.37,1.85) node[anchor=north west] {$\delta_{r-1}$};
\draw (4.95,1.5) node[anchor=north west] {$\delta_r$};
\draw (3.32,1.5) node[anchor=north west] {$\ldots$};
\begin{scriptsize}
\draw [fill=black] (1,1.5) circle (1pt);
\draw [fill=black] (1.5,1.5) circle (1pt);
\draw [fill=black] (2,1.5) circle (1pt);
\draw [fill=black] (2.5,1.5) circle (1pt);
\draw [fill=black] (3,1.5) circle (1pt);
\draw [fill=black] (4,1.5) circle (1pt);
\draw [fill=black] (4.5,1.5) circle (1pt);
\draw [fill=black] (5,1.5) circle (1pt);
\end{scriptsize}
\end{tikzpicture}
    \caption{\label{fig: vertex cofactors}Vertex cofactors along a horizontal T $l$-edge}
\end{figure}
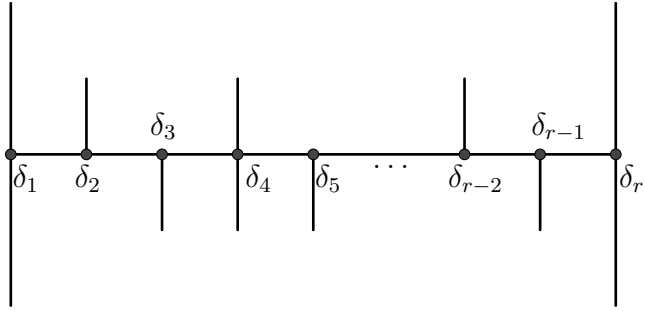

This equation is equivalent to a linear system (denoted by $\mathscr{P}=0$):
\begin{equation}
\begin{pmatrix}
	1 & 1 & \cdots & \cdots & 1 \\
	x_1 & x_2 & \cdots & \cdots & x_r\\
	x_1^2 & x_2^2 & \cdots & \cdots & x_r^2\\
	\cdots & \cdots & \cdots & \cdots & \cdots\\
	x_1^{d-1} & x_2^{d-1} & \cdots & \cdots & x_r^{d-1}\\
	x_1^d & x_2^d & \cdots & \cdots & x_r^d
\end{pmatrix}
\begin{pmatrix}
	\delta_1 \\
	\delta_2 \\
	\delta_3 \\
	\vdots\\
	\delta_{r-1} \\
	\delta_{r} 
\end{pmatrix}
=\begin{pmatrix}
	0 \\
	0 \\
	0 \\
	\vdots\\
	0 \\
	0
\end{pmatrix}.
\label{gcc}
\end{equation}

The linear systems arising from all T $l$-edges in $\mathscr{T}$ constitute the \textbf{conformality conditions} for $S_d(\mathscr{T})$. The conformality vector space of the T-connected component is defined as follows~\cite{huang2024stability,zeng2015}.
\begin{definition}[{\cite{huang2024stability}}]\label{def2.4}
Let $T(\mathscr{T})$ be the T-connected component of $\mathscr{T}$, with edges $l_1,\dots,l_t$ and vertex cofactors $\delta_1,\dots,\delta_v$. The conformality vector space (\textsf{CVS}) is
$$\textsf{CVS}[T(\mathscr{T})] := \bigl\{ \boldsymbol{\delta}=(\delta_1,\dots,\delta_v) : \mathscr{P}_{l_i}=0,\ 1\leq i\leq t \bigr\}.$$
where $\mathscr{P}_{l_i}=0$ denotes the conformality condition across the T $l$-edge $l_i$, analogous to equation~\eqref{gcc}.

The coefficient matrix of this homogeneous system is the \textbf{conformality matrix} $M(T(\mathscr{T}))$.
\end{definition}
The dimension formula is then given by the following formula~\cite{zeng2015}.
\begin{theorem}[{\cite{huang2024stability,zeng2015}}]\label{thm: dim formula}
For a T-mesh $\mathscr{T}$,
$$\dim S_d(\mathscr{T}) = (d+1)^2 + c(d+1) + n_v - \mathrm{rank}\bigl(M(T(\mathscr{T}))\bigr),$$
where $c$ is the number of cross-cuts and $n_v$ is the number of interior vertices.
\end{theorem}
Thus, the study of $\dim S_d(\mathscr T)$ comes down to the study of $\mathrm{rank}(M(L(\mathscr{T})))$ (see Figure~\ref{fig:T-connected component}). Theorem~\ref{thm: dim formula} provides a complete computational framework consisting of the following steps:
\begin{itemize}
\item Classify interior $l$-edges into cross-cuts, rays, and T $l$-edges.
\item Convert continuity constraints across each T $l$-edge into linear relations among vertex cofactors.
\item Assemble local conformality conditions into the global system using the T-connected component $T(\mathscr{T})$.
\item Compute the spline space dimension as $\dim S_d(\mathscr T)=(d+1)^2 + c(d+1) + n_v - \mathrm{rank}\bigl(M(T(\mathscr{T}))\bigr)$.
\end{itemize}

\begin{figure}[htbp]
    \centering
\begin{tikzpicture}[line cap=round,line join=round,>=triangle 45,x=1.0cm,y=1.0cm]
\draw [line width=1.pt] (6.,1.)-- (6.,7.);
\draw [line width=1.pt] (6.,1.)-- (1.,1.);
\draw [line width=1.pt] (1.,1.)-- (1.,7.);
\draw [line width=1.pt] (1.,7.)-- (6.,7.);
\draw [line width=1.pt] (1.,6.)-- (6.,6.);
\draw [line width=1.pt] (2.,7.)-- (2.,1.);
\draw [line width=1.pt] (5.,1.)-- (5.,6.);
\draw [line width=1.pt] (2.,4.)-- (5.,4.);
\draw [line width=1.pt] (1.,2.)-- (5.,2.);
\draw [line width=1.pt] (3.,2.)-- (3.,6.);
\draw [line width=1.pt] (4.,4.)-- (4.,2.);
\draw [line width=1.pt] (2.,3.)-- (5.,3.);
\draw [line width=1.pt] (9.,4.)-- (12.,4.);
\draw [line width=1.pt] (9.,3.)-- (12.,3.);
\draw [line width=1.pt] (10.,2.)-- (10.,6.);
\draw [line width=1.pt] (11.,4.)-- (11.,2.);
\draw (3.3,0.68) node[anchor=north west] {$\mathscr T$};
\draw (10.,0.68) node[anchor=north west] {$T(\mathscr T)$};
\begin{scriptsize}
\draw [fill=black] (2.,4.) circle (1.5pt);
\draw[color=black] (2.2,4.2) node {$v_2$};
\draw [fill=black] (5.,4.) circle (1.5pt);
\draw[color=black] (5.2,4.2) node {$v_5$};
\draw [fill=black] (3.,2.) circle (1.5pt);
\draw[color=black] (3.25,2.2) node {$v_{10}$};
\draw [fill=black] (3.,6.) circle (1.5pt);
\draw[color=black] (3.,6.2) node {$v_1$};
\draw [fill=black] (4.,4.) circle (1.5pt);
\draw[color=black] (4.2,4.2) node {$v_4$};
\draw [fill=black] (4.,2.) circle (1.5pt);
\draw[color=black] (4.25,2.2) node {$v_{11}$};
\draw [fill=black] (2.,3.) circle (1.5pt);
\draw[color=black] (2.2,3.2) node {$v_6$};
\draw [fill=black] (5.,3.) circle (1.5pt);
\draw[color=black] (5.2,3.2) node {$v_9$};
\draw [fill=black] (3.,4.) circle (1.5pt);
\draw[color=black] (3.2,4.2) node {$v_3$};
\draw [fill=black] (3.,3.) circle (1.5pt);
\draw[color=black] (3.2,3.2) node {$v_7$};
\draw [fill=black] (4.,3.) circle (1.5pt);
\draw[color=black] (4.2,3.2) node {$v_8$};
\draw [fill=black] (9.,4.) circle (1.5pt);
\draw[color=black] (9.12,4.2) node {$v_2$};
\draw [fill=black] (12.,4.) circle (1.5pt);
\draw[color=black] (12.,4.2) node {$v_5$};
\draw [fill=black] (9.,3.) circle (1.5pt);
\draw[color=black] (9.12,3.2) node {$v_6$};
\draw [fill=black] (12.,3.) circle (1.5pt);
\draw[color=black] (12.,3.2) node {$v_9$};
\draw [fill=black] (10.,2.) circle (1.5pt);
\draw[color=black] (10.25,2.2) node {$v_{10}$};
\draw [fill=black] (10.,6.) circle (1.5pt);
\draw[color=black] (10.,6.2) node {$v_1$};
\draw [fill=black] (11.,4.) circle (1.5pt);
\draw[color=black] (11.15,4.2) node {$v_4$};
\draw [fill=black] (11.,2.) circle (1.5pt);
\draw[color=black] (11.25,2.2) node {$v_{11}$};
\draw [fill=black] (10.,4.) circle (1.5pt);
\draw[color=black] (10.17,4.2) node {$v_3$};
\draw [fill=black] (10.,3.) circle (1.5pt);
\draw[color=black] (10.17,3.2) node {$v_7$};
\draw [fill=black] (11.,3.) circle (1.5pt);
\draw[color=black] (11.15,3.2) node {$v_8$};
\end{scriptsize}
\end{tikzpicture}
    \caption{\label{fig:T-connected component}T-mesh and its T-connected component}
\end{figure}
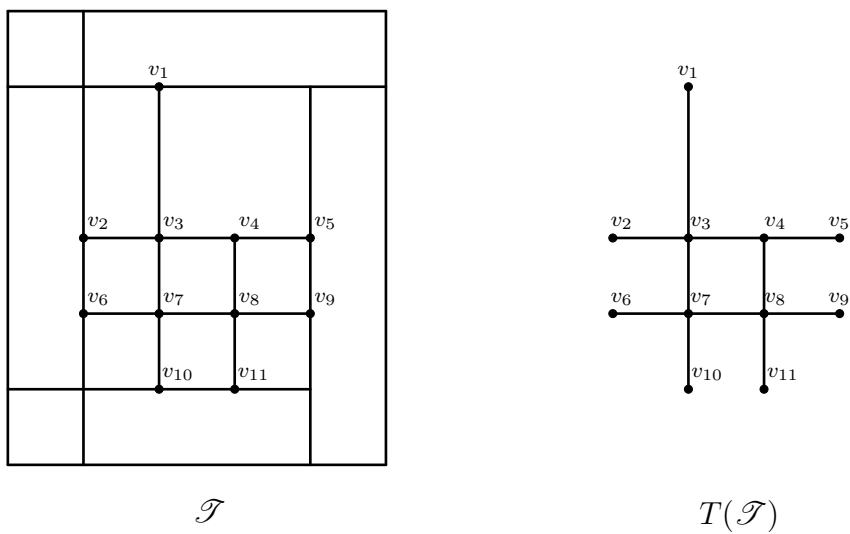

\subsection{Global conformality conditions for spline spaces over rectilinear partitions}
The smoothing cofactor method was originally developed to compute the dimension of spline spaces over arbitrary partitions, including rectilinear partitions. This subsection reviews basic notions of rectilinear partitions and the global conformality conditions used in the smoothing cofactor method.

Roughly speaking, a rectilinear partition $\Delta$ is a partition of a domain formed by a finite set of straight line segments. Examples include triangulations, quadrangulations, and quasi-cross-cut partitions. Faces, edges, and vertices are defined in the usual way. Each interior line segment that cannot be extended further is called an interior large edge, or \textbf{$l$-edge} for short. Following~\cite{li2011dimensions}, interior $l$-edges fall into three categories:
\begin{definition}[{\cite{li2011dimensions}}]
For a rectilinear partition $\Delta$, an interior $l$-edge $e$ is classified as follows:
\begin{itemize}
\item \emph{Cross-cut}: both endpoints lie on the domain boundary;
\item \emph{Ray}: exactly one endpoint lies on the domain boundary;
\item \emph{Truncated $l$-edge}: neither endpoint lies on the domain boundary.
\end{itemize}
\end{definition}

Analogous to the T-connected component in T-meshes, we introduce the following concept.

\begin{definition}
The \textbf{TE-connected component} of a rectilinear partition $\Delta$, denoted $TE(\Delta)$, is the set of all truncated $l$-edges in $\Delta$.
\end{definition}

For example, consider the Morgan--Scott partition $\Delta_{\mathrm{MS}}$ in Figure~\ref{fig:MS partition} (left). This is a classic rectilinear partition. The segments $CD$, $CE$, $DE$, $DA$, $DF$, $FA$, $FB$, $FE$, and $EB$ are $l$-edges. Among them, $CD$, $CE$, $DA$, $FA$, $EB$, and $FB$ are rays, while $DE$, $DF$, and $EF$ are truncated $l$-edges. The set of these truncated $l$-edges forms the TE-connected component $TE(\Delta_{\mathrm{MS}})$, as shown in Figure~\ref{fig:MS partition} (right).

\begin{figure}[htbp]
    \centering
    \begin{tikzpicture}[line cap=round,line join=round,>=triangle 45,x=1cm,y=1cm,scale=0.8]
\draw [line width=1pt] (1,1)-- (8,2);
\draw [line width=1pt] (8,2)-- (3,7);
\draw [line width=1pt] (3,7)-- (1,1);
\draw [line width=1pt] (3,7)-- (3.1,4.4);
\draw [line width=1pt] (3.1,4.4)-- (4.78,4.06);
\draw [line width=1pt] (4.78,4.06)-- (3.9,2.76);
\draw [line width=1pt] (3.9,2.76)-- (3.1,4.4);
\draw [line width=1pt] (4.78,4.06)-- (3,7);
\draw [line width=1pt] (3.9,2.76)-- (1,1);
\draw [line width=1pt] (1,1)-- (3.1,4.4);
\draw [line width=1pt] (3.9,2.76)-- (8,2);
\draw [line width=1pt] (8,2)-- (4.78,4.06);
\draw [line width=1pt] (12,5)-- (13.68,4.66);
\draw [line width=1pt] (13.68,4.66)-- (12.8,3.36);
\draw [line width=1pt] (12.8,3.36)-- (12,5);
\draw (3.9,0.8) node[anchor=north west] {$\Delta_{\mathrm{MS}}$};
\draw (12,0.8) node[anchor=north west] {$TE(\Delta_{\mathrm{MS}})$};
\begin{scriptsize}
\draw [fill=black] (1,1) circle (2pt);
\draw[color=black] (0.75,1.) node {$A$};
\draw [fill=black] (8,2) circle (2pt);
\draw[color=black] (8.16,2.3) node {$B$};
\draw [fill=black] (3,7) circle (2pt);
\draw[color=black] (3.,7.3) node {$C$};
\draw [fill=black] (3.1,4.4) circle (2pt);
\draw[color=black] (3.3,4.75) node {$D$};
\draw [fill=black] (4.78,4.06) circle (2pt);
\draw[color=black] (4.94,4.25) node {$E$};
\draw [fill=black] (3.9,2.76) circle (2pt);
\draw[color=black] (3.95,3.15) node {$F$};
\draw [fill=black] (12,5) circle (2pt);
\draw[color=black] (12.,5.39) node {$D$};
\draw [fill=black] (13.68,4.66) circle (2pt);
\draw[color=black] (13.92,5.05) node {$E$};
\draw [fill=black] (12.8,3.36) circle (2pt);
\draw[color=black] (13.,3.25) node {$F$};
\end{scriptsize}
\end{tikzpicture}

    \caption{\label{fig:MS partition}Morgan-Scott partition and its TE-connected component}
\end{figure}
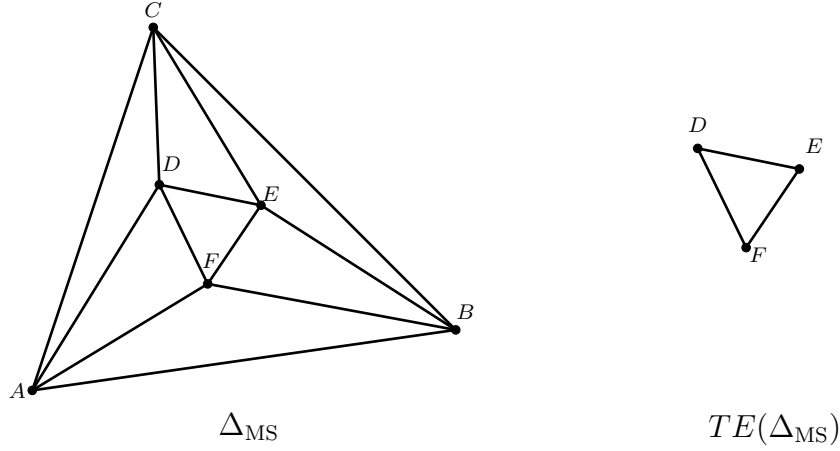

The spline space over a rectilinear partition $\Delta$ is defined as
$$S_d^{\mu}(\Delta) = \bigl\{ f \in C^{\mu}(\Psi) : f|_{\psi} \in \mathbb{P}_d,\ \forall \psi \in \mathscr{C} \bigr\},$$
where $\mathscr{C}$ is the set of all polygonal cells, $\Psi$ is the domain occupied by $\mathscr{C}$ and $\mathbb{P}_d$ is the space of polynomials of total degree at most $d$.
We now review key concepts of the smoothing cofactor method for such spline spaces; further details appear in~\cite{wang1975structural,wang2013multivariate,chui1983smooth}. By Bézout's theorem, the smoothness condition across an edge is equivalent to an algebraic relation, leading to the notion of \textbf{edge cofactor}.

\begin{definition}[\cite{wang1975structural}]
Let $\psi_1$ and $\psi_2$ be adjacent cells sharing an edge $e_{1,2}$ with linear equation $L_{1,2}(x,y) \in \mathbb{P}_1$. Then $f \in C^{\mu}(\overline{\psi_1 \cup \psi_2})$ if and only if there exists $q_{1,2}(x,y) \in \mathbb{P}_{d-\mu-1}$ such that
$$f|_{\psi_1} - f|_{\psi_2} = q_{1,2} \, L_{1,2}^{\mu+1}.$$
The polynomial $q_{1,2}$ is the \textbf{edge cofactor} of $e$ from $\psi_2$ to $\psi_1$.
\end{definition}

We now state the local conformality condition at an interior vertex and the global conformality conditions for the spline space $S_d^{\mu}(\Delta)$.

\begin{definition}\label{def conformality condition}
Let $A$ be an interior vertex of the rectilinear partition $\Delta$. For each interior edge $e_{i,j}$ incident to $A$ and shared by cells $\psi_i$ and $\psi_j$, orient the edge so that a counterclockwise traversal centered at $A$ crosses from $\psi_j$ to $\psi_i$.

\begin{itemize}
    \item The \textbf{local conformality condition} at $A$ is
    \begin{equation}\label{eq:local-conformality}
        \sum_{A} q_{i,j} L_{i,j}^{\mu+1} \equiv 0,
    \end{equation}
    where the sum runs over all interior edges incident to $A$, $q_{i,j}(x,y)$ is the smooth cofactor on $e_{i,j}$, and $L_{i,j}$ is the linear polynomial defining the line containing $e_{i,j}$.

    \item Let $A_v$, $v=1,\dots,V$, be all interior vertices. The \textbf{global conformality conditions} for $S_d^{\mu}(\Delta)$ are
    \begin{equation}\label{eq:global-conformality}
        \sum_{A_v} q_{i,j} L_{i,j}^{\mu+1} \equiv 0, \quad v=1,\dots,V.
    \end{equation}
\end{itemize}
\end{definition}

\begin{remark}
    By Definition~\ref{def conformality condition}, we have 
    $$q_{i,j}=-q_{j,i}.$$
\end{remark}

\begin{example}
Consider vertex $F$ in the Morgan--Scott partition $\Delta_{\mathrm{MS}}$ (Figure~\ref{fig:MS partition}). Let $\psi_i$, $i=1,2,3,4$, be the adjacent cells, with edge orientations and cofactors as shown in Figure~\ref{fig:conformality condition}. By Definition~\ref{def conformality condition}, the local conformality condition at $F$ is
\begin{equation*}
    q_{1,2} L_{DF}^{\mu+1} + q_{2,3} L_{EF}^{\mu+1} + q_{3,4} L_{BF}^{\mu+1} + q_{4,1} L_{FA}^{\mu+1} \equiv 0,
\end{equation*}
where $L_{DF} = L_{1,2}$ is the linear polynomial defining edge $DF$, and similarly for the others.
\end{example}

\begin{figure}[htbp]
    \centering
    \begin{tikzpicture}[line cap=round,line join=round,>=triangle 45,x=1cm,y=1cm,scale=0.9]
\draw [line width=1pt] (0.75,1)-- (8.16,2.3);
\draw [line width=1pt] (3.3,4.75)-- (4.94,4.25);
\draw [line width=1pt] (4.94,4.25)-- (3.95,3.15);
\draw [line width=1pt] (3.95,3.15)-- (3.3,4.75);
\draw [line width=1pt] (3.3,4.75)-- (0.75,1);
\draw [line width=1pt] (3.95,3.15)-- (0.75,1);
\draw [line width=1pt] (4.94,4.25)-- (8.16,2.3);
\draw [line width=1pt] (3.95,3.15)-- (8.16,2.3);
\draw [->,line width=1pt] (5.25,2.25) -- (5.8,3.5); 
\draw [->,line width=1pt] (5.28,3.54) -- (4.15,4.2);
\draw [->,line width=1pt] (4.,4.) -- (2.8,3.1); 
\draw [->,line width=1pt] (2.35,2.65) -- (3.32,2); 
\draw (4.,2.65) node[anchor=north west] {$\psi_4$};
\draw (5.8,3.5) node[anchor=north west] {$\psi_3$};
\draw (3.5,4.65) node[anchor=north west] {$\psi_2$};
\draw (2.25,3.25) node[anchor=north west] {$\psi_1$};
\begin{scriptsize}
\draw [fill=black] (0.75,1) circle (1.5pt);
\draw[color=black] (0.65,1.2) node {$A$};
\draw [fill=black] (8.16,2.3) circle (1.5pt);
\draw[color=black] (8.25,2.5) node {$B$};
\draw [fill=black] (3.3,4.75) circle (1.5pt);
\draw[color=black] (3.4,5.) node {$D$};
\draw [fill=black] (4.94,4.25) circle (1.5pt);
\draw[color=black] (5.08,4.5) node {$E$};
\draw [fill=black] (3.95,3.15) circle (1.5pt);
\draw[color=black] (4.05,3.5) node {$F$};
\draw[color=black] (5.18,3.15) node {$q_{3,4}$};
\draw[color=black] (4.6,3.5) node {$q_{2,3}$};
\draw[color=black] (3.3,4) node {$q_{1,2}$};
\draw[color=black] (3.45,2.5) node {$q_{4,1}$};
\end{scriptsize}
\end{tikzpicture}

    \caption{\label{fig:conformality condition}Local conformality condition around $F$}
\end{figure}
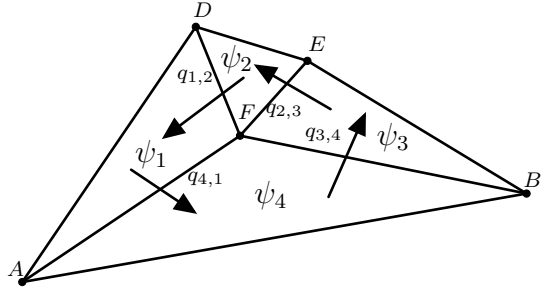

The following theorem characterizes membership in $S_d^{\mu}(\Delta)$ using edge cofactors and global conformality~\cite{wang1975structural,wang2013multivariate,chui1983multivariate}.

\begin{theorem}[\cite{wang1975structural,wang2013multivariate,chui1983multivariate}]\label{thm global conformality conditions for interior vertices}
A function $s$ belongs to $S_d^{\mu}(\Delta)$ if and only if there exist smooth cofactors $q_{i,j}$ for all interior edges $e_{i,j}$ such that the global conformality conditions \eqref{eq:global-conformality} hold for all interior vertices $A_v$, $v=1,\dots,V$.
\end{theorem}

Global conformality conditions are more complex than local ones because, in general rectilinear partitions, an interior edge may connect two interior vertices. For instance, in the Morgan--Scott partition, edge $DE$ is incident to vertices $D$ and $E$, so the edge cofactor of $DE$ denoted as $q_{DE}$ appears in the local conformality conditions at both $D$ and $E$. This coupling increases the difficulty of enforcing global conformality conditions.

\section{A framework for calculating dimensions of spline spaces over rectilinear partitions}\label{section3}
This section extends the framework for dimension calculation of spline spaces over T-mesh to rectilinear partitions via a decoupling strategy. Specifically, the global conformality conditions for the spline space $  S_d^\mu(\Delta)  $ are decomposed into local conformality conditions at each interior vertex together with cofactor constraints along truncated $  l  $-edges. As a result, the dimension of the spline space can be reduced to the analysis of the corresponding conformality spaces over the TE-connected components of the rectilinear partition.

Two central issues must be addressed in this extension:
\begin{itemize}
    \item How to transform the global conformality conditions into equivalent conditions formulated over the TE-connected components;
    \item How to handle conformality conditions for arbitrary smoothness orders $  \mu  $ (in particular, when $  \mu \neq d-1  $).
\end{itemize}

We resolve these issues in the following manner. First, building on the global conformality conditions at all interior vertices presented in Theorem~\ref{thm global conformality conditions for interior vertices}, we introduce the new notion of decoupled cofactors. Using this concept, the original global conformality conditions are equivalently reformulated as a set of decoupled cofactor based global conformality conditions, one for each TE-connected component.
Second, when the smoothness order $  \mu  $ is strictly less than the maximum possible value $  d-1  $, the resulting conformality conditions on decoupled cofactors over TE-connected component are no longer represented by a single numerical matrix. To overcome this difficulty, we apply a homogeneous polynomial decomposition. This approach further decomposes the decoupled cofactor conformality conditions into independent conformality conditions on each homogeneous component. Each of these component-wise conditions can then be expressed as a computable numerical matrix.
Through these two steps, we establish a complete and systematic framework for computing the dimension of spline spaces over rectilinear partitions at arbitrary smoothness orders $  \mu  $. The relevant concepts, theorems, and computational procedures are presented in detail in the subsequent subsections.

\subsection{Global conformality conditions for TE-connected components}
We propose a decoupling strategy that replaces each coupled edge cofactor in the global conformality conditions with independent cofactors at its endpoints, supplemented by consistency constraints.

\begin{definition}\label{decoupled edge cofactor}
Let $e$ be an interior edge of the rectilinear partition $\Delta$ with endpoints $A_1$ and $A_2$, and let $q_e$ be its cofactor in the global conformality conditions \eqref{eq:global-conformality}. The \textbf{decoupled edge cofactor} of $e$ is defined as follows:
\begin{itemize}
\item If one endpoint is on the boundary, say $A_1$ is on the boundary, then $q_e^{(A_2)} := q_e$.
\item If both endpoints are interior vertices, the decoupled cofactor is the pair $(q_e^{(A_1)}, q_e^{(A_2)})$, where $q_e^{(A_i)}$ is the contribution of $e$ to the local conformality condition at $A_i$, $i=1,2$.
\end{itemize}
\end{definition}

\begin{example}
In the Morgan--Scott partition (Figure~\ref{fig:MS partition}),
\begin{itemize}
    \item For edges $CD$, $CE$, $AD$, $AF$, $BF$, $BE$, each has one boundary endpoint, so their decoupled cofactors are $q_{CD}^{(D)}$, $q_{CE}^{(E)}$, $q_{AD}^{(D)}$, $q_{AF}^{(F)}$, $q_{BF}^{(F)}$, and $q_{BE}^{(E)}$;
    \item edges $DE$, $DF$, $EF$ have two interior endpoints, so their decoupled cofactors are the pairs $(q_{DE}^{(D)}, q_{DE}^{(E)})$, $(q_{DF}^{(D)}, q_{DF}^{(F)})$, and $(q_{EF}^{(E)}, q_{EF}^{(F)})$, respectively.
\end{itemize}
\end{example}

\begin{lemma}\label{lem: decoupled charcterastic}
If an interior edge $e$ has two interior endpoints $A_1$ and $A_2$, then
$$
q_e^{(A_1)} = - q_e^{(A_2)}.
$$
Thus, one of the decoupled cofactors equals to $q_e$ and the other is $-q_e$.
\end{lemma}

\begin{proof}
The local conformality condition at each vertex is defined using a counterclockwise orientation around that vertex (Definition~\ref{def conformality condition} and Figure~\ref{fig:opposite edge cofactor}). For the edge $e$, the counterclockwise direction at $A_1$ is opposite to that at $A_2$. Therefore, the two contributions $q_e^{(A_1)}$ and $q_e^{(A_2)}$ have opposite signs but the same magnitude as the original global cofactor $q_e$. This yields $q_e^{(A_1)} = - q_e^{(A_2)}$.

$\Box$
\end{proof}

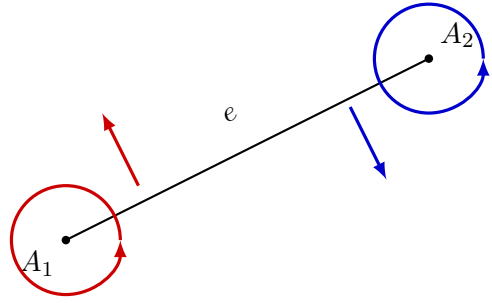
\begin{figure}[htbp]
\centering
\begin{tikzpicture}[thick, >=latex, font=\small, scale=0.8]

  \draw (-2.8,-1.4) -- (3.2,1.6) node[midway,above=8pt,sloped] {$e$};

  \fill (-2.8,-1.4) circle(2pt) node[below left=0pt] {$A_1$};
  \fill (3.2,1.6)   circle(2pt) node[above right=0pt] {$A_2$};

  \draw[red!80!black, very thick, ->] 
        (-2.8,-1.4) ++(0:0.9) arc[start angle=0, end angle=360, radius=0.9];

  \draw[->, red!80!black, line width=1.2pt]
        (-2.8,-1.4) ++(1.2,0.9) -- ++(-0.6,1.2);

  \draw[blue!80!black, very thick, ->] 
        (3.2,1.6) ++(0:0.9) arc[start angle=0, end angle=360, radius=0.9];

  \draw[->, blue!80!black, line width=1.2pt]
        (3.2,1.6) ++(-1.3,-0.8) -- ++(0.6,-1.2);

\end{tikzpicture}
\caption{\label{fig:opposite edge cofactor}Opposite local orientations of edge $e$ at endpoints $A_1$ and $A_2$.}

\end{figure}

By Lemma~\ref{lem: decoupled charcterastic}, the decoupled cofactors split each shared edge cofactor $q_e$ into contributions at its two endpoints. For edges with both interior endpoints, the cofactors at the two ends are negatives of each other, but both are fully determined by the original $q_e$. Thus, Definition~\ref{decoupled edge cofactor} is well-defined.

We now recall the dimension of the conformality space at a single interior vertex. The following result gives the dimension when all incident edges have distinct directions~\cite{schumaker1979dimension}.

\begin{lemma}[\cite{schumaker1979dimension}]\label{lem:local dimension}
Let an interior vertex $A$ be surrounded by $n$ edges defined by distinct lines $L_i(x,y)=\alpha_i x + \beta_i y = 0$, $i=1,\dots,n$, with $\alpha_{i_1}\beta_{i_2} - \alpha_{i_2}\beta_{i_1} \neq 0$ for $i_1 \neq i_2$. The conformality vector space at $A$ is
$$
\textsf{CVS}[A] := \left\{(q_1,\dots,q_n) : \sum_{i=1}^n q_i L_i^{\mu+1} \equiv 0,\ q_i \in \mathbb{P}_{d-\mu-1}\right\}.
$$
Then
$$
\dim \textsf{CVS}[A] = \sum_{j=1}^{d-\mu} \bigl(n(d-\mu-j+1) - (d-j+2)\bigr)_+,
$$
where $u_+ = \max\{0,u\}$.
\end{lemma}

Lemma~\ref{lem:local dimension} assumes distinct edge directions. When collinear edges cross $A$, the following corollary applies.

\begin{corollary}\label{cor:local dimension}
Suppose $2m \le n$ edges around vertex $A$ consist of $m$ pairs of collinear edges with the same line equation ($L_{2i-1} = L_{2i}$, $i=1,\dots,m$) and $n-2m$ edges with distinct directions. Define
$$
W[A] := \left\{(q_1+q_2,\dots,q_{2m-1}+q_{2m},q_{2m+1},\dots,q_n) : \sum_{i=1}^n q_i L_i^{\mu+1} \equiv 0,\ q_i \in \mathbb{P}_{d-\mu-1}\right\}.
$$
Then
$$
\dim W[A] = \sum_{j=1}^{d-\mu} \bigl((n-m)(d-\mu-j+1) - (d-j+2)\bigr)_+.
$$
\end{corollary}

\begin{proof}
The condition $\sum_{i=1}^n q_i L_i^{\mu+1} \equiv 0$ becomes
$$
\sum_{i=1}^m (q_{2i-1}+q_{2i}) L_{2i}^{\mu+1} + \sum_{i=2m+1}^n q_i L_i^{\mu+1} \equiv 0.
$$
Set $p_i = q_{2i-1}+q_{2i}$ ($i=1,\dots,m$) and $p_{m+i} = q_{2m+i}$ ($i=1,\dots,n-2m$). The equation reduces to a conformality condition with $n-m$ distinct lines. Applying Lemma~\ref{lem:local dimension} yields the result.

$\Box$
\end{proof}

Using Corollary~\ref{cor:local dimension}, we now derive the additional conformality constraint imposed on the decoupled cofactors along a truncated $l$-edge.

\begin{lemma}\label{lem global condition in one truncated line}
Let $l$ be a truncated $l$-edge in $\Delta$ with vertices $A_1,\dots,A_r$ ordered along $l$, where $A_1$ and $A_r$ are the endpoints of $l$ (Figure~\ref{fig:global condition in one truncated line}). For each interior segment $A_iA_{i+1}$ ($i=1,\dots,r-1$), let the decoupled cofactor be $(q_i^{(A_i)}, q_i^{(A_{i+1})})$. Define
$$
p_i^{(A_i)} = q_{i-1}^{(A_i)} + q_i^{(A_i)}, \quad i=1,\dots,r,
$$
with the convention $q_0^{(A_1)} := 0$ and $q_r^{(A_r)} := 0$. The global conformality condition along $l$ is
\begin{equation}\label{eq:global condition in one truncated line}
    \sum_{i=1}^r p_i^{(A_i)} \equiv 0.
\end{equation}
\end{lemma}

\begin{proof}
By Corollary~\ref{cor:local dimension}, the local conformality condition at each $A_i$ determines the sum $p_i^{(A_i)}$. Summing over all vertices on $l$ gives
$$
\sum_{i=1}^r p_i^{(A_i)} = q_0^{(A_1)} + \sum_{i=1}^{r-1} (q_i^{(A_i)} + q_i^{(A_{i+1})}) + q_r^{(A_r)}.
$$
The boundary terms vanish, and Lemma~\ref{lem: decoupled charcterastic} yields $q_i^{(A_i)} = -q_i^{(A_{i+1})}$ for each interior segment. Thus, the sum is identically zero.

$\Box$
\end{proof}

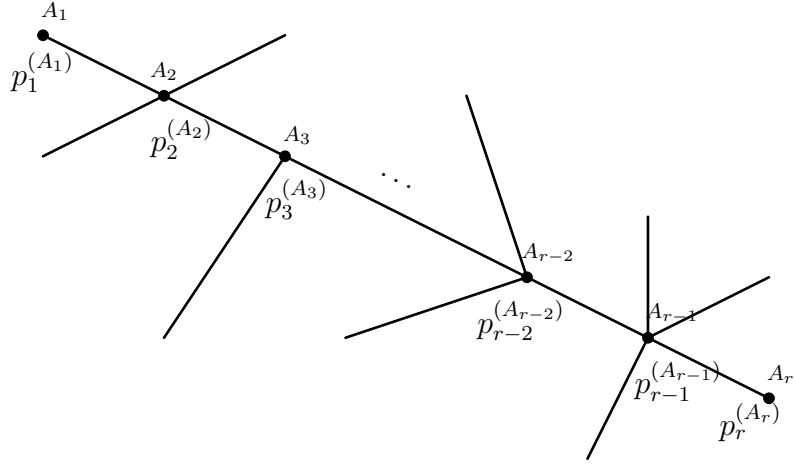
\begin{figure}
    \centering
    \begin{tikzpicture}[line cap=round,line join=round,>=triangle 45,x=1cm,y=1cm,scale=0.8]
\draw [line width=1pt] (0,6)-- (12,0);
\draw [line width=1pt] (0,4)-- (4,6);
\draw [line width=1pt] (4,4)-- (2,1);
\draw [line width=1pt] (8,2)-- (7,5);
\draw [line width=1pt] (8,2)-- (5,1);
\draw [line width=1pt] (10,1)-- (12,2);
\draw [line width=1pt] (10,1)-- (10,3);
\draw [line width=1pt] (10,1)-- (9,-1);
\draw (-0.7,5.9) node[anchor=north west] {$p_{1}^{(A_1)}$};
\draw (1.6,4.8) node[anchor=north west] {$p_{2}^{(A_2)}$};
\draw (3.5,3.8) node[anchor=north west] {$p_{3}^{(A_3)}$};
\draw (7.,1.8) node[anchor=north west] {$p_{r-2}^{(A_{r-2})}$};
\draw (9.6,0.8) node[anchor=north west] {$p_{r-1}^{(A_{r-1})}$};
\draw (11.,0.1) node[anchor=north west] {$p_{r}^{(A_{r})}$};
\draw (5.5,4.) node[anchor=north west, rotate=-26.565] {\dots};
\begin{scriptsize}
\draw [fill=black] (0,6) circle (2.5pt);
\draw[color=black] (0.2,6.4) node {$A_1$};
\draw [fill=black] (12,0) circle (2.5pt);
\draw[color=black] (12.2,0.4) node {$A_r$};
\draw [fill=black] (2,5) circle (2.5pt);
\draw[color=black] (2.,5.4) node {$A_2$};
\draw [fill=black] (4,4) circle (2.5pt);
\draw[color=black] (4.2,4.35) node {$A_3$};
\draw [fill=black] (8,2) circle (2.5pt);
\draw[color=black] (8.35,2.4) node {$A_{r-2}$};
\draw [fill=black] (10,1) circle (2.5pt);
\draw[color=black] (10.4,1.4) node {$A_{r-1}$};
\end{scriptsize}
\end{tikzpicture}

    \caption{\label{fig:global condition in one truncated line}Global conformality condition for a truncated $l$-edge.}
    
\end{figure}

Lemma~\ref{lem global condition in one truncated line} states that solving the local conformality conditions independently at each vertex introduces one linear constraint per truncated $l$-edge.

\begin{definition}
Let $TE(\Delta)$ be the TE-connected component of $\Delta$, consisting of truncated $l$-edges $l_1,\dots,l_t$. The conformality vector space of $TE(\Delta)$ is
$$
\textsf{CVS}[TE(\Delta)] := \bigl\{ \boldsymbol{p} = (p_1^{(A_1)},\dots,p_v^{(A_v)}) : \mathscr{Q}_{l_i}(\boldsymbol{p}) \equiv 0,\ i=1,\dots,t \bigr\},
$$
where $\mathscr{Q}_{l_i} \equiv 0$ is the global conformality condition \eqref{eq:global condition in one truncated line} along $l_i$ and $v$ is the number of vertices on $TE(\Delta)$.
\end{definition}

The global conditions $\mathscr{Q}_{l_i}(\boldsymbol{p}) \equiv 0,\ i=1,\dots,t$ form a homogeneous linear system. Its coefficient matrix is called the \textbf{conformality matrix}, and denoted by $M(TE(\Delta))$.

The main result is the following dimension formula.

\begin{theorem}\label{thm:main}
For any rectilinear partition $\Delta$,
\begin{equation}\label{eq:dimension formula}
\dim S_d^{\mu}(\Delta) = \binom{d+2}{2} + c\binom{d-\mu+1}{2} + \sum_{i=1}^V k_d^{\mu}(N_i) - \mathrm{rank}( M(TE(\Delta))),    
\end{equation}
where
\begin{itemize}
    \item $c$ is the number of cross-cuts,
    \item $V$ is the number of interior vertices,
    \item $N_i$ is the number of distinct edge directions incident to interior vertex $A_i$,
    \item $k_d^{\mu}(N_i) = \sum_{j=1}^{d-\mu} \bigl(N_i(d-\mu-j+1) - (d-j+2)\bigr)_+$.
\end{itemize}
\end{theorem}

\begin{proof}
Let $\mathbb{H}$ be the space of all edge cofactors of $\Delta$ that satisfy the global conformality conditions \eqref{eq:global-conformality}. Then $S_d^\mu(\Delta) = \mathbb{P}_d \oplus \mathbb{H}$, $\mathbb{P}_d$ has defined in Page 4, when introducing the spline space $S_d^{\mu}(\Delta)$. so
$$\dim S_d^\mu(\Delta) = \binom{d+2}{2} + \dim\mathbb{H}.$$
To compute $\dim\mathbb{H}$, consider the decoupled edge cofactors at each interior vertex $A_v$, $v=1,\dots,V$. By Corollary~\ref{cor:local dimension}, the local conformality conditions at $A_v$ yield $\dim W[A_v] = k_d^\mu(N_v)$ free parameters $p_i^{(A_v)}$. Lemma~\ref{lem global condition in one truncated line} shows that these parameters must satisfy additional linear constraints \eqref{eq:global condition in one truncated line} along each truncated $l$-edge in $TE(\Delta)$. The total number of independent constraints equals $\mathrm{rank} M(TE(\Delta))$.

The remaining freedom arises from the structure of lines in $TE(\Delta)$:
\begin{itemize}
    \item Truncated $l$-edges: All edge cofactors are uniquely determined by the $p_i^{(A_i)}$ at vertices on the line; no extra parameters are required.
    \item Rays: One endpoint is an interior vertex, so all cofactors are determined by the $p_i^{(A_i)}$ on the ray; no extra parameters are needed.
    \item Cross-cuts (lines with both endpoints on the boundary): The constraints provide only the sums of adjacent cofactors. Determining all cofactors on each cross-cut requires one additional free parameter per cross-cut. With $c$ cross-cuts, this contributes $c\binom{d-\mu+1}{2}$ degrees of freedom.
\end{itemize}
Thus,
$$\dim\mathbb{H} = c\binom{d-\mu+1}{2} + \sum_{v=1}^V k_d^\mu(N_v) - \mathrm{rank}( M(TE(\Delta))).$$
The statement follows directly from the above equation.
$\Box$
\end{proof}

The dimension formula \eqref{eq:dimension formula} for spline spaces over an arbitrary rectilinear partition $\Delta$ shows that the dimension depends only on the rank of the conformality matrix associated with the TE-connected component $\textsf{CVS}[TE(\Delta)]$.
For the highest smoothness $\mu=d-1$, this conformality matrix forms a numerical matrix. For general $\mu$, however, constructing the conformality matrix and computing its rank is nontrivial. The next subsection addresses this construction and rank evaluation.

\subsection{Solving the global conformality conditions}
This subsection describes the construction of the conformality matrix $M(TE(\Delta))$ for general smoothness order $\mu$. The matrix arises from the global conformality constraints along each truncated $l$-edge in $TE(\Delta)$, which depend on the solution of the local conformality conditions \eqref{eq:local-conformality} at each interior vertex.
We first recall two results from~\cite{schumaker1979dimension}.
\begin{lemma}[\cite{schumaker1979dimension}]\label{lem:homogeneous polynomial}
Define
$$\textsf{CVS}_j[A] := \bigl\{(q_1,\dots,q_n) : \sum_{i=1}^n q_i L_i^{\mu+1} \equiv 0,\ q_i \in \tilde{\mathbb{P}}_{d-\mu-j}\bigr\},$$
where $j=1,\dots,d-\mu$ and $\tilde{\mathbb{P}}_k$ is the space of homogeneous polynomials of degree $k$. Then
$$\textsf{CVS}[A] = \bigoplus_{j=1}^{d-\mu} \textsf{CVS}_j[A].$$
\end{lemma}
\begin{lemma}[\cite{schumaker1979dimension}]\label{lem:local linear equations and full row rank matrix}
For each $j$, $\textsf{CVS}_j[A]$ is the null space of $Q{\mathbf c}=0$, where $Q = (Q_1,\dots,Q_n) \in \mathbb{R}^{(d-j+2) \times n(d-\mu-j+1)}$ and each block $Q_i$ is the matrix
\begin{equation}\label{eq: Q_i}
    Q_i=\begin{pmatrix}
      1 & 0 & 0 & \cdots & 0\\
      \binom{d-j+1}{1}\gamma_i & 1 & 0 & \cdots & 0\\
      \binom{d-j+1}{2}\gamma_i^2 & \binom{d-j}{1}\gamma_i & 1 & \cdots & 0\\
      \vdots & \vdots & \vdots & \ddots & \vdots\\
      \binom{d-j+1}{d-j}\gamma_i^{d-j} & \binom{d-j}{d-j-1}\gamma_i^{d-j-1} & \cdots & \cdots & 1\\
      \binom{d-j+1}{d-j+1}\gamma_i^{d-j+1} & \binom{d-j}{d-j}\gamma_i^{d-j} & \cdots & \cdots & \binom{\mu+1}{\mu+1}\gamma_i^{\mu+1}\\
  \end{pmatrix}_{(d-j+2)\times(d-\mu-j+1)}
\end{equation}  
with $\gamma_i =\frac{\alpha_i}{ \beta_i}$, here $\alpha_i,\beta_i$ are defined as Lemma~\ref{lem:local dimension}. The matrix $Q$ has full rank, i.e. $\mathrm{rank}(Q)=\min\{n(d-\mu-j+1),d-j+2\}$.
\end{lemma}
The corresponding result for the linear space $W[A]$ is given below.
\begin{corollary}\label{cor:local linear equation}
Suppose at vertex $A$ there are $n$ incident edges, of which $2m \le n$ form $m$ pairs of opposite collinear edges with the same line equation $L_{2i-1} = L_{2i}$ ($i=1,\dots,m$), while the remaining $n-2m$ edges have distinct directions. For $j=1,\dots,d-\mu$, define
$$W_j[A] := \bigl\{(q_1+q_2,\dots,q_{2m-1}+q_{2m},q_{2m+1},\dots,q_n) : \sum_{i=1}^n q_i L_i^{\mu+1} \equiv 0,\ q_i \in \tilde{\mathbb{P}}_{d-\mu-1}\bigr\}.$$
Then
\begin{itemize}
\item $W[A] = \bigoplus_{j=1}^{d-\mu} W_j[A]$.
\item Each $W_j[A]$ is the null space of the linear system $Q'{\mathbf b}=0$, where $Q' = (Q_1,\dots,Q_{n-m}) \in \mathbb{R}^{(d-j+2) \times (n-m)(d-\mu-j+1)}$ is built from Lemma~\ref{lem:local linear equations and full row rank matrix} by merging the $m$ pairs of opposite columns. The matrix $Q'$ has full row rank.
\end{itemize}
\end{corollary}
\begin{proof}
The decomposition follows directly from Lemma~\ref{lem:homogeneous polynomial}. The structure of $Q'$ and its full row rank follow from Lemma~\ref{lem:local linear equations and full row rank matrix} after column merging corresponding to collinear opposite edges, leaving $n-m$ distinct directions.
\end{proof}
\begin{lemma}\label{lem:homogeneous global conformality conditions}
Let $l$ be a truncated $l$-edge with vertices $A_1,\dots,A_r$ ordered along $l$. Write each $p_i^{(A_i)} = \sum_{j=1}^{d-\mu} p_{ij}^{(A_i)}$ with $p_{ij}^{(A_i)} \in \tilde{\mathbb{P}}_{d-\mu-j}$. Then
\begin{equation}\label{eq:homogeneous global conformality conditions}
    \sum_{i=1}^r p_i^{(A_i)} \equiv 0 \quad \Longleftrightarrow \quad \sum_{i=1}^r p_{ij}^{(A_i)} \equiv 0 \quad \text{for all } j=1,\dots,d-\mu.
\end{equation}
Thus the global conformality condition \eqref{eq:global condition in one truncated line} decouples into $d-\mu$ independent homogeneous conditions $\sum_{i=1}^r p_{ij}^{(A_i)} \equiv 0$.
\end{lemma}
\begin{proof}
Immediate from the uniqueness of homogeneous decomposition.
\end{proof}
\begin{lemma}\label{lem:equation of homogeneous conditions}
Let $l$ be a truncated $l$-edge with vertices $A_1,\dots,A_r$. For each interior vertex $A_i$ on $l$ and each $j$, let $\{m_l^{(i)}\}_{l=1}^{k(N_i)}$ be a basis of $W_j[A_i]$, where $k(N_i) = (N_i(d-\mu-j) - (d-j+2))_+$. Then
$$p_{ij}^{(A_i)} = \sum_{s=0}^{d-\mu-j} \Bigl( \sum_{l=1}^{k(N_i)} h_{sl}^{(i)} m_l^{(i)} \Bigr) x^s y^{d-\mu-j-s}.$$
The homogeneous global conformality condition for degree $j$ along $l$ becomes the $(d-\mu-j+1)$ scalar equations
$$\sum_{i=1}^r \sum_{l=1}^{k(N_i)} h_{sl}^{(i)} m_l^{(i)} = 0, \quad s=0,\dots,d-\mu-j.$$
\end{lemma}
\begin{proof}
    By Lemma~\ref{lem:homogeneous global conformality conditions}, the independent homogeneous conditions $$\sum_{i=1}^r p_{ij}^{(A_i)} \equiv 0$$ can be represented as 
    $$\sum_{i=1}^r\left(\sum_{s=0}^{d-\mu-j} \Bigl( \sum_{l=1}^{k(N_i)} h_{sl}^{(i)} m_l^{(i)} \Bigr) x^s y^{d-\mu-j-s}\right)\equiv 0.$$
    It follows that 
    $$\sum_{i=1}^r \sum_{l=1}^{k(N_i)} h_{sl}^{(i)} m_l^{(i)} = 0, \quad s=0,\dots,d-\mu-j.$$ This establishes the desired result and thus proves the lemma.
\end{proof}

\begin{theorem}
The system $M(TE(\Delta)) \mathbf{h} = 0$ represents the global conformality conditions, where the matrix $M(TE(\Delta))$ is assembled from the linear equations derived by applying Lemma \ref{lem:equation of homogeneous conditions} to every truncated $l$-edge in $TE(\Delta)$. Specifically, for each homogeneous degree $j = 1, \dots, d-\mu$, the matrix encodes the scalar constraints:$$\sum_{i=1}^r \sum_{l=1}^{k(N_i)} h_{sl}^{(i)} m_l^{(i)} = 0, \quad s=0, \dots, d-\mu-j,$$where $\mathbf{h}$ is the vector consisting of the coefficients $h_{sl}^{(i)}$.
\end{theorem}

The construction of $M(TE(\Delta))$ is summarized in Algorithm~\ref{alg:matrix}.

\begin{algorithm}[h]
\caption{\label{alg:matrix}Construction of the conformality matrix $M(TE(\Delta))$}
\begin{enumerate}
\item Identify the set $TE(\Delta)$ of all truncated $l$-edges and extractable from $\Delta$.
\item For each interior vertex $A_i$ and each $j=1,\dots,d-\mu$, compute a basis $\{m_l^{(i)}\}_{l=1}^{k(N_i)}$ of $W_j[A_i]$.
\item For each truncated $l$-edge in $TE(\Delta)$ and each $j$, generate the $(d-\mu-j+1)$ linear equations
$$\sum_{i=1}^r \sum_{l=1}^{k(N_i)} h_{sl}^{(i)} m_l^{(i)} = 0, \quad s=0,\dots,d-\mu-j.$$
\item The matrix $M(TE(\Delta))$ is assembled from the linear constraints defined by Step 3.
\end{enumerate}
\end{algorithm}

\section{Partitions with Disjoint Truncated $l$-edges}

In this section, we investigate a specific class of partitions where the $TE$-connected components consist of disjoint truncated $l$-edges. We demonstrate that under certain conditions, the dimension of the resulting spline space can be derived through a simplified application of our dimension calculation framework.

\begin{definition}
    A rectilinear partition $\Delta$ is called a \textbf{partition with disjoint truncated $l$-edges} if its $TE$-connected component, denoted by $TE(\Delta)$, is formed by several pairwise disjoint truncated $l$-edges.
\end{definition}

For the spline space $S_d^{\mu}(\Delta)$ over such partitions, we establish the following theorem.

\begin{theorem}\label{thm:dimension of disjoint L-edges partition}
    Let $\Delta$ be a partition with disjoint truncated $l$-edges, where $t$ denotes the number of such $l$-edges. Suppose that for each interior vertex $A_i$ ($i=1, \dots, V$), the number of distinct edge directions $N_i$ satisfies $N_i \ge \mu + 3$. Then, the dimension of the spline space $S_d^{\mu}(\Delta)$ is given by
    \begin{equation}\label{eq:dimension formula special partition}
    \dim S_d^{\mu}(\Delta) = \binom{d+2}{2} + (c-t)\binom{d-\mu+1}{2} + \sum_{i=1}^V k_d^{\mu}(N_i),
    \end{equation}
    where $c$, $V$, and $k_d^{\mu}(N_i)$ are defined as in Theorem~\ref{thm:main}.
\end{theorem}

\begin{proof}
    Let the $TE$-connected component $TE(\Delta)$ be composed of $t$ disjoint truncated $l$-edges, denoted by $l_1, l_2, \dots, l_t$. Due to the disjoint nature of these components, the conformality matrix $M(TE(\Delta))$ exhibits a block-diagonal structure. Consequently, the rank of the global conformality matrix is the sum of the ranks of the individual blocks:
    \[ \text{rank}(M(TE(\Delta))) = \sum_{i=1}^t \text{rank}(M(TE(l_i))), \]
    where $M(TE(l_i))$ is the conformality matrix associated with the $l$-edge $l_i$ under the global conformality conditions.

    Without loss of generality, consider a single truncated $l$-edge $l_i$ with vertices $A_1, \dots, A_r$. The assumption $N_i \ge \mu+3$ implies that for all $j \in \{1, \dots, d-\mu\}$, the following inequality holds:
    \[ N_i \ge 2 + \frac{\mu+1}{d-\mu-j+1}, \]
    which is algebraically equivalent to
    \[ N_i(d-\mu-j+1) - (d-j+2) \ge d-\mu-j+1. \]
    According to Lemma~\ref{lem:homogeneous polynomial}, this condition ensures that the dimension of the solution space for the cofactor space $\textsf{CVS}_j[A_i]$ at each vertex exceeds $d-\mu-j+1$. Since the number of degrees of freedom for each homogeneous component of the edge cofactor in $\textsf{CVS}_j[A_i]$ is exactly $d-\mu-j+1$, we can designate these homogeneous components as free variables. Under this setting, the global conformality condition for the $l$-edge $l_i$ can be expressed as the following polynomial identity:
    \begin{equation}\label{eq:sum_zero}
    \sum_{i=1}^r p_i^{(A_i)} \equiv 0,
    \end{equation}
    where $p_i^{(A_i)} \in \mathbb{P}_{d-\mu-1}$ for $i=1, \dots, r$. 

    The dimension of the space of $r$-tuples $(p_1, \dots, p_r)$ satisfying \eqref{eq:sum_zero} is $(r-1)\binom{d-\mu+1}{2}$. Since the total number of coefficients in the $r$ polynomials is $r\binom{d-\mu+1}{2}$, the number of independent linear constraints imposed by \eqref{eq:sum_zero} is:
    \[ \text{rank}(M(TE(l_i))) = r\binom{d-\mu+1}{2} - (r-1)\binom{d-\mu+1}{2} = \binom{d-\mu+1}{2}. \]
    Summing over all $t$ disjoint $l$-edges, we obtain $\text{rank}(M(TE(\Delta))) = t\binom{d-\mu+1}{2}$. Substituting this into the general dimension formula from Theorem~\ref{thm:main} completes the proof.
\end{proof}

\begin{example}\label{exm:separated l edges partition}
    Consider the spline space $S_d^1(\Delta)$ over the partition $\Delta$ illustrated in Figure~\ref{fig:example_partition} ($d\ge 2$). The partition contains two disjoint truncated $l$-edges, $AB$ and $DE$, thus qualifying as a partition with disjoint truncated $l$-edges. For all interior vertices $A, B, C, D, E$, we have $N_i = 4$, satisfying the condition $N_i \ge \mu + 3 = 4$.
    
    Applying Theorem~\ref{thm:dimension of disjoint L-edges partition} with $c=0$, $t=2$, $\mu=1$, and $V=5$, we have:
    \[ \dim S_d^1(\Delta) = \binom{d+2}{2} - 2\binom{d}{2} + 5 k_d^1(4). \]
    Substituting the specific values, the dimension evaluates to:
    \[ \dim S_d^1(\Delta) = 7d^2 - 15d + 11. \]
\end{example}

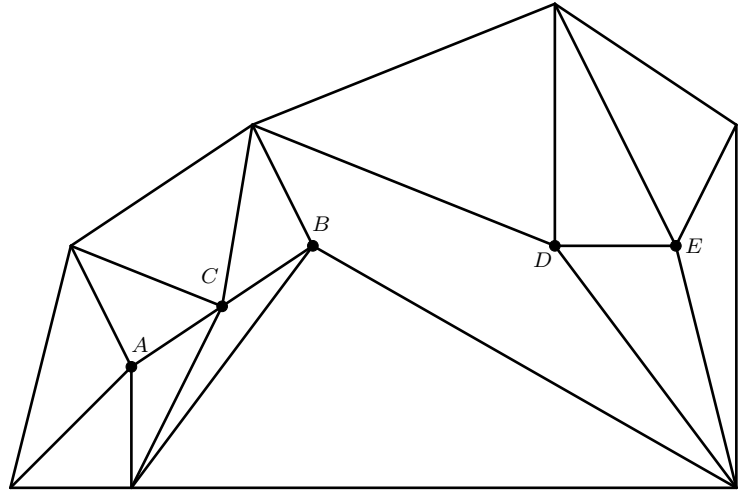
\begin{figure}[htbp]
    \centering
    \begin{tikzpicture}[line cap=round,line join=round,>=triangle 45,x=1.0cm,y=1.0cm,scale=0.8]
        \draw [line width=1.pt] (2,2) -- (5,4);
        \draw [line width=1.pt] (2,2) -- (0,0);
        \draw [line width=1.pt] (2,2) -- (2,0);
        \draw [line width=1.pt] (2,2) -- (1,4);
        \draw [line width=1.pt] (1,4) -- (3.5,3);
        \draw [line width=1.pt] (3.5,3) -- (2,0);
        \draw [line width=1.pt] (1,4) -- (0,0);
        \draw [line width=1.pt] (3.5,3) -- (4,6);
        \draw [line width=1.pt] (1,4) -- (4,6);
        \draw [line width=1.pt] (5,4) -- (4,6);
        \draw [line width=1.pt] (5,4) -- (2,0);
        \draw [line width=1.pt] (5,4) -- (12,0);
        \draw [line width=1.pt] (0,0) -- (12,0);
        \draw [line width=1.pt] (9,4) -- (11,4);
        \draw [line width=1.pt] (4,6) -- (9,4);
        \draw [line width=1.pt] (12,6) -- (11,4);
        \draw [line width=1.pt] (12,6) -- (12,0);
        \draw [line width=1.pt] (11,4) -- (12,0);
        \draw [line width=1.pt] (9,4) -- (12,0);
        \draw [line width=1.pt] (4,6) -- (9,8);
        \draw [line width=1.pt] (9,8) -- (12,6);
        \draw [line width=1.pt] (9,8) -- (9,4);
        \draw [line width=1.pt] (9,8) -- (11,4);
        \begin{scriptsize}
        \draw [fill=black] (2.,2.) circle (2.5pt);
        \draw[color=black] (2.14,2.37) node {$A$};
        \draw [fill=black] (5.,4.) circle (2.5pt);
        \draw[color=black] (5.14,4.37) node {$B$};
        \draw [fill=black] (3.5,3.) circle (2.5pt);
        \draw[color=black] (3.3,3.5) node {$C$};
        \draw [fill=black] (9.,4.) circle (2.5pt);
        \draw[color=black] (8.8,3.77) node {$D$};
        \draw [fill=black] (11.,4.) circle (2.5pt);
        \draw[color=black] (11.3,4.) node {$E$};
        \end{scriptsize}
    \end{tikzpicture}
    \caption{A rectilinear partition $\Delta$ featuring two disjoint truncated $l$-edges $AB$ and $DE$.}
    \label{fig:example_partition}
\end{figure}

Finally, we prove that the dimension of the spline space satisfying the conditions in Theorem~\ref{thm:dimension of disjoint L-edges partition} is equal to the lower bound established for arbitrary triangulations in~\cite{schumaker1979dimension}. As our partitions are rectilinear and not necessarily triangulated, we first define a formal lower bound for such spaces.

\begin{definition}
    For any spline space over a rectilinear partition $\Delta$, the \textbf{formal lower bound} is defined as:
    \begin{equation}\label{dimension: lower bound}
        \text{Low}(S_d^{\mu}(\Delta))=\binom{d+2}{2}+\binom{d-\mu+1}{2}E-\left[\binom{d+2}{2}-\binom{\mu+2}{2}\right]V+\sum\limits_{i=1}^V\sigma_i,
    \end{equation}
    where $E$ is the number of interior edges, $V$ is the number of interior vertices, and $\sigma_i=\sum\limits_{j=1}^{d-\mu}(\mu+1+j(1-N_i))_{+}$ with $N_i$ being the number of distinct edge directions incident to interior vertex $A_i$.
\end{definition}

\begin{remark}
    Schumaker~\cite{schumaker1979dimension} demonstrated that \eqref{dimension: lower bound} is the lower bound for triangulations. For general rectilinear partitions, we utilize this expression as a formal lower bound for comparative analysis.
\end{remark}

We now show that under the condition $N_i\ge\mu+3$, the dimension formula \eqref{eq:dimension formula special partition} is identical to \eqref{dimension: lower bound}. First, we establish a topological relationship between the partition parameters.

\begin{lemma}\label{lem:topology relationship}
    For any rectilinear partition $\Delta$ with disjoint truncated $l$-edges, the following identity holds:
    \begin{equation}
        E - \sum_{i=1}^V N_i = c - t.
    \end{equation}
\end{lemma}

\begin{proof}
Let the sets of all cross-cuts, rays and truncated $l$-edges in \(\Delta\) be \(C(\Delta)\), \(R(\Delta)\), and \(TE(\Delta)\) respectively, with \(c = |C(\Delta)|\), \(r = |R(\Delta)|\), and \(t = |TE(\Delta)|\).

For each $l$-edge \(l_k\), let \(v(l_k)\) be the number of interior vertices on \(l_k\), and let \(e(l_k)\) be the number of interior edge segments on \(l_k\). Then we have:
\begin{enumerate}
    \item If \(l_k\) is a cross-cut, then \(e(l_k) = v(l_k) + 1\);
    \item If \(l_k\) is a ray, then \(e(l_k) = v(l_k)\);
    \item If \(l_k\) is a truncated $l$-edge, then \(e(l_k) = v(l_k) - 1\).
\end{enumerate}

The total number of interior edges is
\[
E = \sum_{l_k \in C(\Delta)} (v(l_k)+1) + \sum_{l_k \in R(\Delta)} v(l_k) + \sum_{l_k \in TE(\Delta)} (v(l_k)-1).
\]
This can be rewritten as
\[
E = \sum_{\text{all } l_k} v(l_k) + c - t.
\]

By double counting the incidences between interior vertices and $l$-edges, we obtain \(\sum_{\text{all } l_k} v(l_k) = \sum_{i=1}^V N_i\), where \(N_i\) is the number of $l$-edges passing through the interior vertex \(A_i\). Therefore,
\[
E = \sum_{i=1}^V N_i + c - t,
\]
which gives the topological identity
\[
E - \sum_{i=1}^V N_i = c - t.
\]
\end{proof}

The following lemmas simplify the terms $\sigma_i$ and $k_d^{\mu}(N_i)$ when $N_i\ge\mu+3$.

\begin{lemma}[\cite{wang2013multivariate}]
    The expression for $k_d^{\mu}(N)$ is:
    \[ k_d^{\mu}(N)=\frac{1}{2} \left( d - \mu - \left\lfloor \frac{\mu + 1}{N-1} \right\rfloor \right)_{+} \cdot \Big( (N-1)d - (N+1)\mu + (N-3) + (N-1) \left\lfloor \frac{\mu + 1}{N-1} \right\rfloor\Big). \]
\end{lemma}

\begin{lemma}\label{lem: sigma kdmu}
    In the case of $N_i\ge \mu+3$, we have:
    \begin{itemize}
        \item $\sigma_i=0$ for $i=1,2,\ldots,V$.
        \item $k_d^{\mu}(N_i)=\frac{1}{2} \left( d - \mu \right) \cdot \Big( (N_i-1)d - (N_i+1)\mu + (N_i-3) \Big)$.
    \end{itemize}
\end{lemma}

\begin{proof}
    \begin{itemize}
        \item For a fixed $i$, since $N_i\ge \mu+3$, then $1-N_i \le -(\mu+2)<0$. For $j \in \{1, \dots, d-\mu\}$, the function $f(j)=\mu+1+j(1-N_i)$ is strictly decreasing. Its maximum occurs at $j=1$, where $f(1) = \mu+2-N_i \le \mu+2-(\mu+3) = -1 < 0$. Consequently, $(\mu+1+j(1-N_i))_{+}=0$ for all $j$, and thus $\sigma_i=0$.
        \item Given $N_i \ge \mu+3$, it follows that $0 \le \frac{\mu+1}{N_i-1} < 1$, hence $\lfloor \frac{\mu + 1}{N_i-1} \rfloor = 0$. Substituting this into the formula for $k_d^{\mu}(N_i)$ yields the result.
    \end{itemize}
\end{proof}

\begin{theorem}\label{thm:low=dim}
    Let $\Delta$ be a partition with disjoint truncated $l$-edges. If the number of distinct edge directions $N_i$ satisfies $N_i \ge \mu + 3$ for all $i$, then:
    \[ \dim S_d^{\mu}(\Delta) = \text{Low}(S_d^{\mu}(\Delta)). \]
\end{theorem}

\begin{proof}
    Consider the difference $\text{Low}(S_d^{\mu}(\Delta)) - \dim S_d^{\mu}(\Delta)$. Using Theorem~\ref{thm:dimension of disjoint L-edges partition} and Lemma~\ref{lem:topology relationship}, we expand the terms:
    \[ \text{Low}(S_d^{\mu}(\Delta)) - \dim S_d^{\mu}(\Delta) = \sum\limits_{i=1}^V \left[ N_i \binom{d-\mu+1}{2} + \sigma_i - \left( \binom{d+2}{2} - \binom{\mu+2}{2} \right) - k_d^{\mu}(N_i) \right]. \]
    By Lemma~\ref{lem: sigma kdmu}, we substitute $\sigma_i=0$ and the simplified $k_d^{\mu}(N_i)$. Verification shows that:
    \[N_i\binom{d-\mu+1}{2}+\sigma_i-\binom{d+2}{2}+\binom{\mu+2}{2}-k_d^{\mu}(N_i)\equiv0\ \ \text{for}\ i=1,2,\ldots,V. \]
    Therefore, the difference is zero, and the dimensions are equal.
\end{proof}

Theorem~\ref{thm:low=dim} demonstrates that for a partition $\Delta$ with disjoint truncated $l$-edges satisfying $N_i \ge \mu+3$, the dimension of the spline space $S_d^{\mu}(\Delta)$ coincides with Schumaker’s lower bound. This result implies that the lower bound is attainable in specific cases for arbitrary degree $d$ and smoothness order $\mu$.

\section{Computational Examples}
In this section, we apply the framework developed in Section~\ref{section3} to two well-known examples. The first is the spline space $S_2^1(\Delta_{\mathrm{MS}})$  over the Morgan--Scott partition. The second is the spline space $S_5^2(\Delta_{\mathrm{YS}})$ over the partition as shown in Figure~\ref{fig:yuan-stillman partition}, which we denote by $\Delta_{\mathrm{YS}}$ and refer to as the Yuan--Stillman partition. This partition was introduced by Yuan and Stillman~\cite{yuan2019counter} as a counterexample to the Schenck--Stiller ``$2r+1$'' conjecture.

\subsection{Spline space $S_2^1(\Delta_{MS})$ over Morgan-Scott partition}
We now illustrate the complete procedure of Algorithm~\ref{alg:matrix} by computing the dimension of the spline space $S_2^1(\Delta_{MS})$ over the Morgan-Scott partition step by step.
\begin{itemize}
    \item[Step 1.] Identify the set $TE(\Delta_{\mathrm{MS}})$ as shown in Figure~\ref{fig:MS partition}. 
    \item[Step 2.] In $S_2^1(\Delta_{\mathrm{MS}})$, the edge cofactors are constants. Following Algorithm~\ref{alg:matrix}, we determine a basis of $W_1[\cdot]$ at each interior vertex $D$, $E$, and $F$. For vertex $D$, the cofactors $q_{DC}$, $q_{DA}$, $q_{DE}$, and $q_{DF}$ satisfy the local conformality condition:
    $$\begin{pmatrix}
        1 & 1 & 1 &1\\ 2\gamma_{DC} & 2\gamma_{DA} & 2\gamma_{DE} & 2\gamma_{DF}\\\gamma_{DC}^2 & \gamma_{DA}^2 & \gamma_{DE}^2 & \gamma_{DF}^2 
    \end{pmatrix}\begin{pmatrix}
        q_{DC}\\q_{DA}\\q_{DE}\\q_{DF}
    \end{pmatrix}=\begin{pmatrix}
        0\\0\\0\\0
    \end{pmatrix}.$$
    The solution is
$$\begin{pmatrix}
q_{DC} \\[6pt]
q_{DA} \\[6pt]
q_{DE} \\[6pt]
q_{DF}
\end{pmatrix}
=
h_D
\begin{pmatrix}
\dfrac{1}{(\gamma_{DC} - \gamma_{DA})(\gamma_{DC} - \gamma_{DE})(\gamma_{DC} - \gamma_{DF})} \\[10pt]
\dfrac{1}{(\gamma_{DA} - \gamma_{DC})(\gamma_{DA} - \gamma_{DE})(\gamma_{DA} - \gamma_{DF})} \\[10pt]
\dfrac{1}{(\gamma_{DE} - \gamma_{DC})(\gamma_{DE} - \gamma_{DA})(\gamma_{DE} - \gamma_{DF})} \\[10pt]
\dfrac{1}{(\gamma_{DF} - \gamma_{DC})(\gamma_{DF} - \gamma_{DA})(\gamma_{DF} - \gamma_{DE})}
\end{pmatrix},
\quad h_D \in \mathbb{R}.$$
Analogous expressions hold for the interior vertices $E$ and $F$.

\item[Step 3.]We now enforce the global conformality conditions along the truncated $l$-edges $DE$, $DF$, and $EF$. For edge $DE$, the condition requires $q_{DE} + q_{ED} = 0$. Substituting the expressions from Step 2 yields
$$\frac{h_D}{(\gamma_{DE} - \gamma_{DC})(\gamma_{DE} - \gamma_{DA})(\gamma_{DE} - \gamma_{DF})}+\frac{h_E}{(\gamma_{ED} - \gamma_{EC})(\gamma_{ED} - \gamma_{EB})(\gamma_{ED} - \gamma_{EF})}=0.$$
The same procedure applies to edges $DF$ and $EF$.

\item[Step 4.] By assembling the three equations obtained above, the conformality matrix $M(TE(\Delta_{\mathrm{MS}}))$ is 
\begin{equation}\label{eq:global conformality conditions of TE}
    M(TE(\Delta_{\mathrm{MS}}))=\begin{pmatrix}
m_{DE}^D & m_{ED}^E & 0 \\[4pt]
m_{DF}^D & 0 & m_{FD}^F \\[4pt]
0 & m_{EF}^E & m_{FE}^F
\end{pmatrix}.
\end{equation}
where $m_{DE}^D:=\frac{1}{(\gamma_{DE} - \gamma_{DC})(\gamma_{DE} - \gamma_{DA})(\gamma_{DE} - \gamma_{DF})}$ and the remaining entries are defined analogously.

\item[Step 5.] The rank of the conformality matrix is
$$\mathrm{rank}\,M(TE(\Delta_{\mathrm{MS}})) = 
\begin{cases}
2 & \text{if } m_{DE}^D m_{FD}^F m_{EF}^E + m_{ED}^E m_{DF}^D m_{FE}^F = 0, \\
3 & \text{if } m_{DE}^D m_{FD}^F m_{EF}^E + m_{ED}^E m_{DF}^D m_{FE}^F \neq 0.
\end{cases}$$
The condition $m_{DE}^D m_{FD}^F m_{EF}^E + m_{ED}^E m_{DF}^D m_{FE}^F = 0$ is equivalent to
$$\underbrace{\frac{(\gamma_{DE}-\gamma_{DA})(\gamma_{DE}-\gamma_{DC})}{(\gamma_{DF}-\gamma_{DA})(\gamma_{DF}-\gamma_{DC})}}_{R_D}
\cdot
\underbrace{\frac{(\gamma_{EF}-\gamma_{EB})(\gamma_{EF}-\gamma_{EC})}{(\gamma_{DE}-\gamma_{EB})(\gamma_{DE}-\gamma_{EC})}}_{R_E}
\cdot
\underbrace{\frac{(\gamma_{FD}-\gamma_{FA})(\gamma_{FD}-\gamma_{FB})}{(\gamma_{EF}-\gamma_{FA})(\gamma_{EF}-\gamma_{FB})}}_{R_F}
= 1.$$
As shown in Appendix~\ref{appendix}, this relation holds if and only if the lines $AE$, $BD$, and $CF$ are concurrent. Therefore,
$$\dim S_2^1(\Delta_{\mathrm{MS}}) = 
\begin{cases}
7 & \text{if } AE,\ BD,\ \text{and}\ CF\ \text{are concurrent}, \\
6 & \text{otherwise}.
\end{cases}$$
    
\end{itemize}


\subsection{Spine space $S_5^2(\Delta_{\mathrm{YS}})$ over Yuan-Stillman partition}
In this subsection, we compute the dimension of the spline space $S_5^2(\Delta_{\mathrm{YS}})$ over the Yuan--Stillman partition $\Delta_{\mathrm{YS}}$ as shown in Figure~\ref{fig:yuan-stillman partition}.

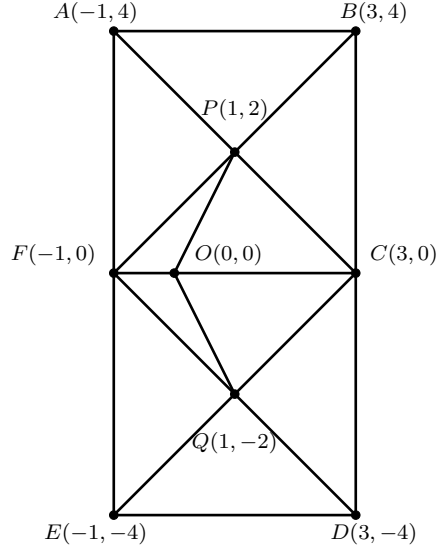
\begin{figure}[ht]
    \centering
    \begin{tikzpicture}[line cap=round,line join=round,>=triangle 45,x=1cm,y=1cm,scale=0.8]
\draw [line width=1pt] (3,4)-- (3,0);     
\draw [line width=1pt] (3,0)-- (3,-4);    
\draw [line width=1pt] (3,-4)-- (-1,-4);  
\draw [line width=1pt] (-1,-4)-- (-1,0);  
\draw [line width=1pt] (-1,0)-- (-1,4);   
\draw [line width=1pt] (-1,4)-- (3,4);    
\draw [line width=1pt] (3,0)-- (-1,0);    
\draw [line width=1pt] (-1,0)-- (3,4);    
\draw [line width=1pt] (3,0)-- (-1,4);    
\draw [line width=1pt] (3,0)-- (-1,-4);   
\draw [line width=1pt] (3,-4)-- (-1,0);   
\draw [line width=1pt] (1,2)-- (0,0);     
\draw [line width=1pt] (0,0)-- (1,-2);    

\begin{scriptsize}
\draw [fill=black] (-1,4) circle (2pt);
\draw[color=black] (-1.3,4.3) node {$A(-1,4)$};
\draw [fill=black] (3,4) circle (2pt);
\draw[color=black] (3.3,4.3) node {$B(3,4)$};
\draw [fill=black] (3,0) circle (2pt);
\draw[color=black] (3.8,0.3) node {$C(3,0)$};
\draw [fill=black] (3,-4) circle (2pt);
\draw[color=black] (3.3,-4.3) node {$D(3,-4)$};
\draw [fill=black] (-1,-4) circle (2pt);
\draw[color=black] (-1.3,-4.3) node {$E(-1,-4)$};
\draw [fill=black] (-1,0) circle (2pt);
\draw[color=black] (-2,0.3) node {$F(-1,0)$};
\draw [fill=black] (0,0) circle (2pt);
\draw[color=black] (0.9,0.3) node {$O(0,0)$};
\draw [fill=black] (1,2) circle (2pt);
\draw[color=black] (1.,2.7) node {$P(1,2)$};
\draw [fill=black] (1,-2) circle (2pt);
\draw[color=black] (1.,-2.8) node {$Q(1,-2)$};
\end{scriptsize}
\end{tikzpicture}
    \caption{\label{fig:yuan-stillman partition}Yuan-Stillman Partition}
\end{figure}

Algorithm~\ref{alg:matrix} goes as follows:
\begin{itemize}
    \item[Step 1.] Identify the set $TE(\Delta_{\mathrm{YS}})$ as shown in Figure~\ref{fig:TE of y-s partition}. 
\begin{figure}[ht]
    \centering
    \begin{tikzpicture}[line cap=round,line join=round,>=triangle 45,x=1cm,y=1cm,scale=0.8]
\draw [line width=1pt] (1,2)-- (0, 0);
\draw [line width=1pt] (0,0)-- (1,-2);
\begin{scriptsize}
\draw [fill=black] (0,0) circle (2pt);
\draw[color=black] (-0.5,0) node {$O$};
\draw [fill=black] (1,2) circle (2pt);
\draw[color=black] (1,2.36) node {$P$};
\draw [fill=black] (1,-2) circle (2pt);
\draw[color=black] (1,-2.36) node {$Q$};
\draw (-0.8,-3.5) node[anchor=north west] {\normalsize$TE(\Delta_{\mathrm{YS}})$};
\end{scriptsize}
\end{tikzpicture}
\caption{\label{fig:TE of y-s partition}The TE-connected component of Yuan-Stillman Partition}
\end{figure}
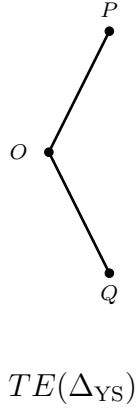

\item[Step 2.] 
In $S_5^2(\Delta_{\mathrm{YS}})$, the edge cofactors are quadratic polynomials. 
Following Algorithm~\ref{alg:matrix}, we compute bases of $W_1[\cdot]$, $W_2[\cdot]$, and $W_3[\cdot]$ at the interior vertices $P$, $Q$, and $O$. It's easy to check that $\dim W_3[\cdot]=0$ for $P$, $Q$, and $O$ by Lemma~\ref{lem:local dimension}.

As an illustration, consider $W_2[P]$. 
By Corollary~\ref{cor:local linear equation}, this space corresponds to the degree-one component of the (decoupled) edge cofactors at $P$.

These components, denoted
$$f_1(q_{PO}^{(P)}), \quad f_1(q_{PB}^{(P)}+q_{PF}^{(P)}), \quad f_1(q_{PA}^{(P)}+q_{PC}^{(P)}),$$
satisfy the local conformality condition:
\begin{equation}\label{eq:conformality condition for edge cofactors in P}
\begin{pmatrix}
1 & 0 & 1 & 0 & 1 & 0 \\
\binom{4}{1} & 1 & -\binom{4}{1} & 1 & -2\binom{4}{1} & 1 \\
\binom{4}{2} & \binom{3}{1} & \binom{4}{2} & -\binom{3}{1} & 4\binom{4}{2} & -2\binom{3}{1} \\
\binom{4}{3} & \binom{3}{2} & -\binom{4}{3} & \binom{3}{2} & -8\binom{4}{3} & 4\binom{3}{2} \\
\binom{4}{4} & \binom{3}{3} & \binom{4}{4} & -\binom{3}{3} & 16\binom{4}{4} & -8\binom{3}{3}
\end{pmatrix}
\begin{pmatrix}
c_{1,0} \\ c_{1,1} \\ c_{2,0} \\ c_{2,1} \\ c_{3,0} \\ c_{3,1}
\end{pmatrix}
=
\begin{pmatrix}
0 \\ 0 \\ 0 \\ 0 \\ 0 \\ 0
\end{pmatrix},
\end{equation}
where 
\begin{align*}
f_1(q_{PA}^{(P)}+q_{PC}^{(P)}) 
&= c_{1,0}(x+y-3) + c_{1,1}(x-1), \\
f_1(q_{PB}^{(P)}+q_{PF}^{(P)}) 
&= c_{2,0}(-x+y-1) + c_{2,1}(x-1), \\
f_1(q_{PO}^{(P)}) 
&= c_{3,0}(-2x+y) + c_{3,1}(x-1).
\end{align*}

The solution of equation~\eqref{eq:conformality condition for edge cofactors in P} is
$$
\begin{pmatrix}
c_{1,0} \\ c_{1,1} \\ c_{2,0} \\ c_{2,1} \\ c_{3,0} \\ c_{3,1}
\end{pmatrix}
= s_1^{(P)} \begin{pmatrix}
-5  \\ 12 \\ -27 \\ 108 \\ 32 \\ 48
\end{pmatrix}, \qquad s_1^{(P)} \in \mathbb{R}.
$$ 
Thus,
\begin{align*}
f_1(q_{PB}^{(P)} + q_{PF}^{(P)}) &= s_1^{(P)} (7x - 5y +3), \\
f_1(q_{PA}^{(P)} + q_{PC}^{(P)}) &= 27s_1^{(P)} (5x - y - 3), \\
f_1(q_{PO}^{(P)}) &= 16s_1^{(P)} (- x+2y  - 3).
\end{align*}

Bases for $W_1[P]$, and for the spaces $W_j[O]$, $W_j[Q]$ ($j=1,2,3$) are obtained in the same manner. 

    \item[Step 3.] We now impose the global conformality conditions along the truncated $l$-edges $PO$ and $OQ$.
Consider edge $PO$ as an example. The condition requires 
$$q_{PO}^{(P)} + q_{OP}^{(O)} \equiv 0.$$ 
By Lemma~\ref{lem:homogeneous global conformality conditions}, this is equivalent to the three independent homogeneous conditions
$$\text{homo}_i(q_{PO}^{(P)}) + \text{homo}_i(q_{OP}^{(O)}) \equiv 0,\ i=0,1,2,$$
where $\text{homo}_i(\cdot)$ denotes the degree-$i$ homogeneous component.

By Step2, we can get
\begin{align*}
    \text{homo}_2\left(q_{PO}^{(P)}\right) &= (4s_{2,1}^{(P)}-2s_{2,2}^{(P)}+s_{2,3}^{(P)})x^2+s_{2,1}^{(P)}y^2+(-4s_{2,1}^{(P)}+s_{2,2}^{(P)})xy,\\
    \text{homo}_1\left(q_{PO}^{(P)}\right) &=
    (2s_{2,2}^{(P)}-2s_{2,3}^{(P)}-16s_1^{(P)})x+(32s_1^{(P)}-s_{2,2}^{(P)})y,\\
    \text{homo}_0\left(q_{PO}^{(P)}\right) &=
    s_{2,3}^{(P)}-48s_1^{(P)},\\
    \text{homo}_2\left(q_{OP}^{(O)}\right) &= 
    (2s_{2,3}^{(O)} - 24s_{2,1}^{(O)})x^2 + (3s_{2,2}^{(O)} - 16s_{2,1}^{(O)})y^2+ (2s_{2,2}^{(O)} - 36s_{2,1}^{(O)} + 3s_{2,3}^{(O)})xy,\\
    \text{homo}_1\left(q_{OP}^{(O)}\right) &=
    2s_1^{(O)}x+3s_1^{(O)}y,\\
    \text{homo}_0\left(q_{PO}^{(P)}\right) &=
    0,
\end{align*}

Applying the homogeneous conditions $\text{homo}_i(q_{PO}^{(P)}) + \text{homo}_i(q_{OP}^{(O)}) \equiv 0$ for $i=0, 1, 2$ yields a system of 6 linear equations with 8 unknowns: $s_{2,1}^{(P)}, s_{2,2}^{(P)}, s_{2,3}^{(P)}, s_{2,1}^{(O)}, s_{2,2}^{(O)}, s_{2,3}^{(O)}, s_1^{(P)},$ and $s_1^{(O)}$. Similarly, imposing the global conformality condition $q_{QO}^{(Q)} + q_{OQ}^{(O)} \equiv 0$ provides the remaining scalar equations, leading to the following system:

$$\left(\begin{array}{cccccccccccc} 0 & 0 & 0 & 0 & 0 & 0 & 0 & 0 & -48 & 0 & 0 & 1\\ 3 & 0 & 0 & 0 & 0 & 0 & 0 & 0 & -32 & 0 & -1 & 0\\ 0 & -16 & 3 & 0 & 0 & 0 & 0 & 0 & 0 & 1 & 0 & 0\\ 2 & 0 & 0 & 0 & 0 & 0 & 0 & 0 & -16 & 0 & -2 & -2\\ 0 & -36 & 2 & 3 & 0 & 0 & 0 & 0 & 0 & 4 & 1 & 0\\ 0 & -24 & 0 & 2 & 0 & 0 & 0 & 0 & 0 & 4 & 2 & 1\\ 0 & 0 & 0 & 0 & -48 & 0 & 0 & 1 & 0 & 0 & 0 & 0\\ -3 & 0 & 0 & 0 & 32 & 0 & -1 & 0 & 0 & 0 & 0 & 0\\ 0 & -16 & -3 & 0 & 0 & 1 & 0 & 0 & 0 & 0 & 0 & 0\\ 2 & 0 & 0 & 0 & -16 & 0 & 2 & -2 & 0 & 0 & 0 & 0\\ 0 & 36 & 2 & -3 & 0 & -4 & 1 & 0 & 0 & 0 & 0 & 0\\ 0 & -24 & 0 & 2 & 0 & 4 & -2 & 1 & 0 & 0 & 0 & 0 \end{array}\right)
\begin{pmatrix}
s_1^{(O)}\\
s_{2,1}^{(O)} \\
s_{2,2}^{(O)} \\
s_{2,3}^{(O)}\\
s_1^{(P)}\\
s_{2,1}^{(P)} \\
s_{2,2}^{(P)} \\
s_{2,3}^{(P)} \\
s_1^{(Q)}\\
s_{2,1}^{(Q)} \\
s_{2,2}^{(Q)} \\
s_{2,3}^{(Q)} 
\end{pmatrix}=\begin{pmatrix}
0\\
0\\
0\\
0 \\
0 \\
0 \\
0 \\
0 \\
0 \\
0 \\
0 \\
0
\end{pmatrix},$$
where $s_{2,i}^{(\cdot)} (i=1,2,3)$ stands for the three freedoms of the degree-two component for interior vertex of $P$, $Q$ and $O$ respectively.  
    
    Numerical computation yields $\mathrm{rank}M(TE(\Delta_{\mathrm{YS}})) = 11$. Therefore,
$$\dim S_5^2(\Delta_{\mathrm{YS}}) = \binom{7}{2} + 5\binom{4}{2} + 4 \cdot 3 - 11 = 52.$$

\end{itemize}

This value is strictly greater than Schumaker’s lower bound of 51~\cite{schumaker1979dimension}, thereby confirming the counterexample of Yuan and Stillman~\cite{yuan2019counter} to the Schenck--Stiller ``\(2r+1\)'' conjecture.

\subsection{Mixed polygonal partition}
In this subsection, we present additional computational examples to further demonstrate the generality and flexibility of our approach. We consider a mixed polygonal partition $\Delta$ consisting of triangles and quadrilaterals, as illustrated in Figure~\ref{fig:mixed partition}. We study the spline space $  S_5^3(\Delta)  $ defined over this partition and compute its dimension by systematically applying Algorithm~\ref{alg:matrix} as follows:

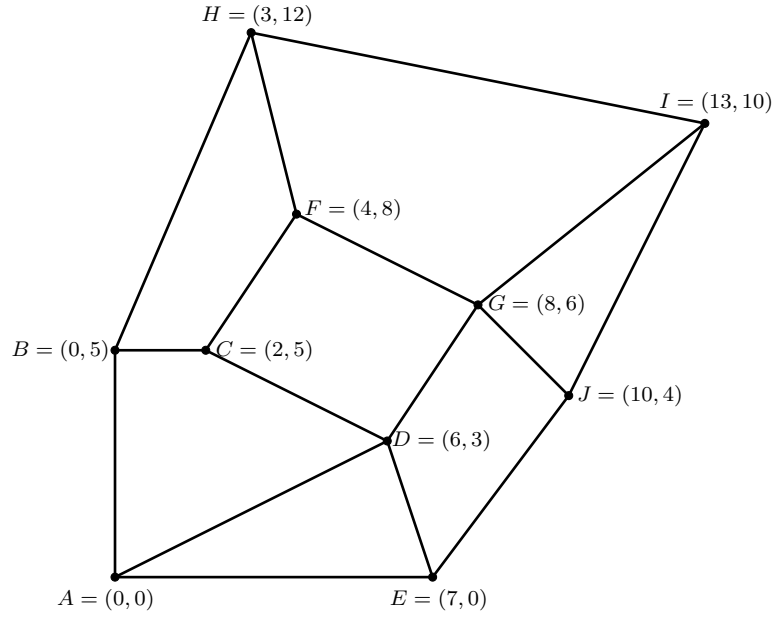
\begin{figure}
    \centering
\begin{tikzpicture}[line cap=round,line join=round,>=triangle 45,x=1.0cm,y=1.0cm,scale=0.6]
\draw [line width=1.pt] (0.,0.)-- (0.,5.);
\draw [line width=1.pt] (0.,5.)-- (2.,5.);
\draw [line width=1.pt] (2.,5.)-- (6.,3.);
\draw [line width=1.pt] (6.,3.)-- (7.,0.);
\draw [line width=1.pt] (7.,0.)-- (0.,0.);
\draw [line width=1.pt] (2.,5.)-- (4.,8.);
\draw [line width=1.pt] (4.,8.)-- (8.,6.);
\draw [line width=1.pt] (8.,6.)-- (6.,3.);
\draw [line width=1.pt] (4.,8.)-- (3.,12.);
\draw [line width=1.pt] (3.,12.)-- (0.,5.);
\draw [line width=1.pt] (8.,6.)-- (10.,4.);
\draw [line width=1.pt] (3.,12.)-- (13.,10.);
\draw [line width=1.pt] (13.,10.)-- (8.,6.);
\draw [line width=1.pt] (13.,10.)-- (10.,4.);
\draw [line width=1.pt] (10.,4.)-- (7.,0.);
\draw [line width=1.pt] (0.,0.)-- (6.,3.);
\begin{scriptsize}
\draw [fill=black] (0.,0.) circle (2.5pt);
\draw[color=black] (-0.2,-0.5) node {$A=(0,0)$};
\draw [fill=black] (0.,5.) circle (2.5pt);
\draw[color=black] (-1.2,5.) node {$B=(0,5)$};
\draw [fill=black] (2.,5.) circle (2.5pt);
\draw[color=black] (3.3,5.) node {$C=(2,5)$};
\draw [fill=black] (6.,3.) circle (2.5pt);
\draw[color=black] (7.2,3.) node {$D=(6,3)$};
\draw [fill=black] (7.,0.) circle (2.5pt);
\draw[color=black] (7.14,-0.5) node {$E=(7,0)$};
\draw [fill=black] (4.,8.) circle (2.5pt);
\draw[color=black] (5.25,8.1) node {$F=(4,8)$};
\draw [fill=black] (8.,6.) circle (2.5pt);
\draw[color=black] (9.3,6.) node {$G=(8,6)$};
\draw [fill=black] (3.,12.) circle (2.5pt);
\draw[color=black] (3.14,12.37) node {$H=(3,12)$};
\draw [fill=black] (10.,4.) circle (2.5pt);
\draw[color=black] (11.35,4.) node {$J=(10,4)$};
\draw [fill=black] (13.,10.) circle (2.5pt);
\draw[color=black] (13.25,10.45) node {$I=(13,10)$};
\end{scriptsize}
\end{tikzpicture}
    \caption{\label{fig:mixed partition}A mixed polygonal partition $\Delta$ consisting of triangles, quadrilaterals, and pentagons}
\end{figure}

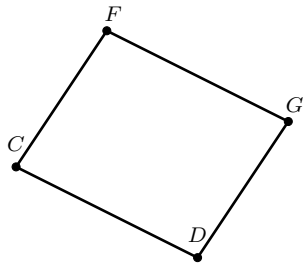
\begin{figure}
    \centering
    \begin{tikzpicture}[line cap=round,line join=round,>=triangle 45,x=1.0cm,y=1.0cm,scale=0.6]
\draw [line width=1.pt] (2.,5.)-- (6.,3.);
\draw [line width=1.pt] (2.,5.)-- (4.,8.);
\draw [line width=1.pt] (4.,8.)-- (8.,6.);
\draw [line width=1.pt] (8.,6.)-- (6.,3.);
\begin{scriptsize}
\draw [fill=black] (2.,5.) circle (2.5pt);
\draw[color=black] (2.,5.5) node {$C$};
\draw [fill=black] (6.,3.) circle (2.5pt);
\draw[color=black] (6.,3.5) node {$D$};
\draw [fill=black] (4.,8.) circle (2.5pt);
\draw[color=black] (4.14,8.37) node {$F$};
\draw [fill=black] (8.,6.) circle (2.5pt);
\draw[color=black] (8.14,6.37) node {$G$};
\end{scriptsize}
\end{tikzpicture}
    \caption{\label{fig:TE of partition}The TE-connected component of the mixed partition $\Delta$ shown in Figure~\ref{fig:mixed partition}.}
\end{figure}

\begin{itemize}
    \item[Step 1.] Identify the set $TE(\Delta)$ as shown in Figure~\ref{fig:TE of partition}. 
    \item[Step 2.] In $S_5^3(\Delta)$, the edge cofactors are linear polynomials. Following Algorithm~\ref{alg:matrix}, we compute bases of $W_1[\cdot]$ and $W_2[\cdot]$ at the interior vertices $C,D,F$ and $G$.
    
    As an illustration, consider $W_1[D]$. 
    By Corollary~\ref{cor:local linear equation}, this space corresponds to the degree-one homogeneous component of the (decoupled) edge cofactors at $C$.

    These components, denoted
$$f_1(q_{DC}^{(D)}), \quad f_1(q_{DA}^{(D)}), \quad f_1(q_{DG}^{(D)}), \quad f_1(q_{DE}^{(D)})$$
satisfy the local conformality condition:
\begin{equation}\label{eq:conformality condition for edge cofactors in D}
\begin{pmatrix}
1 & 0 & 1 & 0 & 1 & 0 & 1 & 0\\
-\frac{1}{2}\binom{5}{1} & 1 & 3\binom{5}{1} & 1 & -\frac{3}{2}\binom{5}{1} & 1 & \frac{1}{2}\binom{5}{1} & 1\\
\frac{1}{4}\binom{5}{2} & -\frac{1}{2}\binom{4}{1} & 3^2\binom{5}{2} & 3\binom{4}{1} & \frac{9}{4}\binom{5}{2} & -\frac{3}{2}\binom{4}{1}& \frac{1}{4}\binom{5}{2} & \frac{1}{2}\binom{4}{1} \\
-\frac{1}{8}\binom{5}{3} & \frac{1}{4}\binom{4}{2} & 3^3\binom{5}{3} & 3^2\binom{4}{2} & -\frac{27}{8}\binom{5}{3} & \frac{9}{4}\binom{4}{2} & \frac{1}{8}\binom{5}{3} & \frac{1}{4}\binom{4}{2} \\
\frac{1}{16}\binom{5}{4} & -\frac{1}{8}\binom{4}{3} & 3^4\binom{5}{4} & 3^3\binom{4}{3} & \frac{81}{16}\binom{5}{4} & -\frac{27}{8}\binom{4}{3} & \frac{1}{16}\binom{5}{4} & \frac{1}{8}\binom{4}{3}\\
-\frac{1}{32}\binom{5}{5} & \frac{1}{16}\binom{4}{4} & 3^5\binom{5}{5} & 3^4\binom{4}{4} & -\frac{243}{32}\binom{5}{5} & \frac{81}{16}\binom{4}{4} & \frac{1}{32}\binom{5}{5} & \frac{1}{16}\binom{4}{4}
\end{pmatrix}
\begin{pmatrix}
c_{1,0} \\ c_{1,1} \\ c_{2,0} \\ c_{2,1} \\ c_{3,0} \\ c_{3,1} \\ c_{4,0} \\ c_{4,1}
\end{pmatrix}
=
\begin{pmatrix}
0 \\ 0 \\ 0 \\ 0 \\ 0 \\ 0
\end{pmatrix},
\end{equation}
where 
\begin{align*}
f_1(q_{DA}^{(D)}) 
&= c_{1,0}\left(-\frac{1}{2}(x-3)+y-6\right) + c_{1,1}(x-3), \\
f_1(q_{DE}^{(D)}) 
&= c_{2,0}(3(x-3)+y-6) + c_{2,1}(x-3), \\
f_1(q_{DG}^{(D)}) 
&= c_{3,0}\left(-\frac{3}{2}(x-3)+y-6\right) + c_{3,1}(x-3),\\
f_1(q_{DC}^{(D)}) 
&= c_{4,0}\left(\frac{1}{2}(x-3)+y-6\right) + c_{4,1}(x-3).
\end{align*}
By solve the equation~\eqref{eq:conformality condition for edge cofactors in D} we can get
\begin{align*}
f_1(q_{DA}^{(D)}) 
&= -\left(\frac{500\,q_{1}}{49}+\frac{120\,q_{2}}{49}\right)\,\left(x-6\right)-\left(\frac{300\,q_{1}}{343}+\frac{212\,q_{2}}{343}\right)\,\left(\frac{x}{2}-y\right), \\[4pt]
f_1(q_{DE}^{(D)}) 
&= \left(\frac{800\,q_{1}}{3969}+\frac{416\,q_{2}}{3969}\right)\,\left(x-6\right)-\left(\frac{14272\,q_{1}}{250047}+\frac{7744\,q_{2}}{250047}\right)\,\left(3\,x+y-21\right), \\[4pt]
f_1(q_{DG}^{(D)}) 
&= -\left(\frac{1325\,q_{1}}{729}+\frac{428\,q_{2}}{729}\right)\,\left(y-\frac{3\,x}{2}+6\right)-\left(\frac{250\,q_{1}}{81}+\frac{85\,q_{2}}{81}\right)\,\left(x-6\right), \\[4pt]
f_1(q_{DC}^{(D)}) 
&= q_{1}\,\left(\frac{x}{2}+y-6\right)+q_{2}\,\left(x-6\right).
\end{align*}
where $q_1, q_2$ are the freedom for the interior vertex $D$.

Bases for $W_2[D]$, and for the spaces $W_j[C],W_j[F], W_j[G]\ (j=1,2)$ are obtained in the same manner.
     \item[Step 3.] Similarly, the global conformality conditions lead to a homogeneous linear system defined by the following matrix::
     $$M(TE(\Delta_{\mathrm{YS}})) = \begin{pmatrix} 
1/2 & 1 & 0 & 0 \\ 
1 & 0 & 0 & 0 \\ 
-6 & -6 & 0 & 0 \\ 
0 & 0 & 1/2 & 1 \\ 
0 & 0 & 1 & 0 \\ 
0 & 0 & -10 & -8 \\ 
-\frac{175}{486} & -\frac{41}{243} & \frac{164}{481} & -\frac{352}{2279} \\ 
-\frac{1325}{729} & -\frac{428}{729} & -\frac{1256}{1095} & \frac{397}{856} \\ 
\frac{1850}{243} & \frac{674}{243} & \frac{457}{110} & -\frac{3713}{2400} 
\end{pmatrix}$$

Since $\mathrm{rank}M(TE(\Delta_{\mathrm{YS}})) = 4$. Therefore,
$$\dim S_5^2(\Delta_{\mathrm{YS}}) = \binom{7}{2} + 6\binom{3}{2} + 4 - 4 = 39.$$
\end{itemize}


\section{Conclusion and future work}

This paper establishes a unified and computationally explicit framework for determining the dimension of bivariate spline spaces $S_d^\mu(\Delta)$ over arbitrary rectilinear partitions via the smoothing cofactor method.

The primary contributions are summarized as follows. First, the dimension formula previously developed for T-meshes is extended to general rectilinear partitions through the introduction of $TE$-connected components. The dimension of $S_d^\mu(\Delta)$ is expressed by an explicit algebraic formula based on the rank of the associated conformality matrix $M(TE(\Delta))$. Second, a systematic algorithm (Algorithm~\ref{alg:matrix}) is provided to construct this matrix for any smoothness order $\mu$, utilizing homogeneous polynomial decomposition combined with decoupled edge cofactors.

Furthermore, a new class of rectilinear partitions, termed \textbf{partitions with disjoint truncated $l$-edges}, is introduced. It is proven that under specific conditions, the dimension of the corresponding spline space exactly attains Schumaker's lower bound~\cite{schumaker1979dimension}.

The effectiveness and correctness of the proposed framework are validated through several classical and novel examples. These include the Morgan--Scott partition, which recovers the well-known dimension instability, and the Yuan--Stillman partition, which confirms a counterexample to the Schenck--Stiller "$2r+1$" conjecture.

By reducing dimension calculation to the rank computation of an explicitly constructible matrix, the framework remains applicable to spline spaces of any degree $d$ and smoothness $\mu$ over arbitrary rectilinear partitions. This provides a practical and theoretically rigorous tool for spline analysis in geometric modeling, finite element methods, and isogeometric analysis. Future research will focus on extending this approach to trivariate splines and developing adaptive basis constructions derived from the established conformality constraints.


\subsection*{Acknowledgements}
We are grateful to Assistant Professor Beihui Yuan (BIMSA) for sharing the numerical results presented in Section 4.2. This comparison served as a crucial validation of our proposed method. This work is supported by the Key Project of the National Natural Science Foundation of China (No.12494550 and No.12494555) and the Fundamental Research Funds for the Central Universities (No.WK0010000096). And the author declares that there is no conflict of interest.

\bibliographystyle{IEEEtran}
\bibliography{T1}

\appendix\label{appendix}
\section{Appendix: Geometric Interpretation and Proof of the Consistency Condition}
In this appendix, we provide a rigorous geometric proof that the algebraic consistency condition derived in the main text implies the concurrency of the characteristic lines $AE$, $BD$, and $CF$. We utilize the trigonometric form of Ceva's Theorem and the Area Method.

\subsection{Lemma: Trigonometric Ceva's Theorem}

\begin{lemma}[Trigonometric Form of Ceva's Theorem]
Let $A, B, C$ be the vertices of a triangle. Let $D, E, F$ be points associated with sides or sectors such that lines $AE$, $BD$, and $CF$ are well-defined. These three lines are concurrent (intersect at a single point $O$) if and only if:
\begin{equation}
    \frac{\sin \angle BAE}{\sin \angle CAE} \cdot \frac{\sin \angle CBF}{\sin \angle ABF} \cdot \frac{\sin \angle ACD}{\sin \angle BCD} = 1.
\end{equation}
\end{lemma}

\subsection{Proof of Concurrency}

We consider the geometric configuration where $D, E, F$ are strictly inside the triangle $\triangle ABC$, connected as shown in Figure~\ref{fig:MS partition}. The consistency condition is given by:
\begin{equation}
    \label{eq:gamma_identity}
    \mathcal{C} = \underbrace{\frac{(\gamma_{DE}-\gamma_{DA})(\gamma_{DE}-\gamma_{DC})}{(\gamma_{DF}-\gamma_{DA})(\gamma_{DF}-\gamma_{DC})}}_{R_D} \cdot
    \underbrace{\frac{(\gamma_{EF}-\gamma_{EB})(\gamma_{EF}-\gamma_{EC})}{(\gamma_{DE}-\gamma_{EB})(\gamma_{DE}-\gamma_{EC})}}_{R_E} \cdot
    \underbrace{\frac{(\gamma_{FD}-\gamma_{FA})(\gamma_{FD}-\gamma_{FB})}{(\gamma_{EF}-\gamma_{FA})(\gamma_{EF}-\gamma_{FB})}}_{R_F}
    = 1.
\end{equation}

\begin{itemize}
    \item[Step 1.] Transformation to Trigonometric Form.

    Using the identity $\gamma_{XY} - \gamma_{XZ} \propto \sin(\theta_{XZ} - \theta_{XY})$, and observing that the cosine terms in the denominators cancel out cyclically (due to the symmetry of the edges $DE, EF, FD$), Eq.~(\ref{eq:gamma_identity}) transforms into a product of sine ratios:
\begin{equation}
    \label{eq:sine_identity}
    \frac{\sin(DA, DE)\sin(DC, DE)}{\sin(DA, DF)\sin(DC, DF)} \cdot 
    \frac{\sin(EB, EF)\sin(EC, EF)}{\sin(EB, ED)\sin(EC, ED)} \cdot 
    \frac{\sin(FA, FD)\sin(FB, FD)}{\sin(FA, FE)\sin(FB, FE)} = 1.
\end{equation}

    \item[Step 2.] Geometric Interpretation via Area Method.

    We relate the sine terms to the areas of the sub-triangles. Let $S_{XYZ}$ denote the area of $\triangle XYZ$. Using the area formula $S_{ADE} = \frac{1}{2} DA \cdot DE \cdot \sin(DA, DE)$, we can rewrite the ratios. For instance, the term $R_D$ (in its sine form) becomes:
\begin{equation}
    R_D \propto \frac{S_{ADE}}{S_{ADF}} \cdot \frac{S_{CDE}}{S_{CDF}}.
\end{equation}
Note that the side lengths $DA, DC$ cancel within the fraction, and the lengths $DE, DF$ will cancel when multiplied with the corresponding terms from $R_E$ and $R_F$.

Substituting these area forms into the full product $\mathcal{C} = R_D R_E R_F = 1$:
\begin{equation}
    \underbrace{\left(\frac{S_{ADE} \cdot S_{CDE}}{S_{ADF} \cdot S_{CDF}}\right)}_{R_D} \cdot
    \underbrace{\left(\frac{S_{BEF} \cdot S_{CEF}}{S_{BDE} \cdot S_{CDE}}\right)}_{R_E} \cdot
    \underbrace{\left(\frac{S_{ADF} \cdot S_{BDF}}{S_{AEF} \cdot S_{BEF}}\right)}_{R_F} = 1.
\end{equation}
Observing this product, we see significant cancellations:
\begin{itemize}
    \item $S_{CDE}$ appears in the numerator of $R_D$ and the denominator of $R_E$.
    \item $S_{ADF}$ appears in the denominator of $R_D$ and the numerator of $R_F$.
    \item $S_{BEF}$ appears in the numerator of $R_E$ and the denominator of $R_F$.
\end{itemize}
Taking into account the symmetry $S_{XYZ} = S_{XZY}$, the equation simplifies to a concise relation involving six specific areas:
\begin{equation}
    \label{eq:area_identity}
    \frac{S_{ADE}}{S_{AEF}} \cdot \frac{S_{BDF}}{S_{BDE}} \cdot \frac{S_{CEF}}{S_{CDF}} = 1.
\end{equation}

    \item[Step 3.] From Area Ratios to Concurrency.

    We now interpret the geometric meaning of Eq.~(\ref{eq:area_identity}). Consider the first term $\frac{S_{ADE}}{S_{AEF}}$. Since $\triangle ADE$ and $\triangle AEF$ share a common base $AE$, their area ratio is equal to the ratio of the heights (distances) from vertices $D$ and $F$ to the line $AE$. Let $h_P(L)$ denote the perpendicular distance from point $P$ to line $L$. Then:
\begin{equation}
    \frac{S_{ADE}}{S_{AEF}} = \frac{h_D(AE)}{h_F(AE)}, \quad
    \frac{S_{BDF}}{S_{BDE}} = \frac{h_F(BD)}{h_E(BD)}, \quad
    \frac{S_{CEF}}{S_{CDF}} = \frac{h_E(CF)}{h_D(CF)}.
\end{equation}
Thus, the condition becomes a constraint on the distances:
\begin{equation}
    \label{eq:height_identity}
    \frac{h_D(AE)}{h_F(AE)} \cdot \frac{h_F(BD)}{h_E(BD)} \cdot \frac{h_E(CF)}{h_D(CF)} = 1.
\end{equation}
We now show that this condition implies concurrency. Assume the lines $AE, BD, CF$ intersect at a common point $O$. In this case, points $D$ and $B$ lie on the line passing through $O$ (line $BD$), so $D, B, O$ are collinear. Similarly for $E, A, O$ and $F, C, O$.
Using similar triangles (or basic trigonometry at the intersection $O$), the ratio of distances to the transversal line $AE$ is determined by the position of $O$:
\begin{equation}
    \frac{h_D(AE)}{h_F(AE)} = \frac{OD \cdot \sin \angle DOA}{OF \cdot \sin \angle FOA}.
\end{equation}
Applying this to all three terms in Eq.~(\ref{eq:height_identity}):
\begin{equation}
    \text{LHS} = \left( \frac{OD \sin \angle DOA}{OF \sin \angle FOA} \right) \cdot \left( \frac{OF \sin \angle FOB}{OE \sin \angle EOB} \right) \cdot \left( \frac{OE \sin \angle EOC}{OD \sin \angle DOC} \right).
\end{equation}
The segment lengths $OD, OF, OE$ cancel out. Furthermore, since $A,E,O$ are collinear, $B,D,O$ are collinear, and $C,F,O$ are collinear, the angles are vertically opposite (e.g., $\angle DOA = \angle EOB$ or $\pi - \angle EOB$). Thus, the sine terms cancel perfectly, resulting in 1.

Since the geometric derivation is reversible (assuming non-degenerate configurations), the satisfaction of the consistency condition (Eq.~\ref{eq:gamma_identity}) implies that Eq.~(\ref{eq:height_identity}) holds, which is the necessary and sufficient condition for the lines $AE$, $BD$, and $CF$ to be concurrent.
\end{itemize}

\end{document}